\documentclass[reqno,a4paper,11pt]{amsart}

\usepackage{graphicx}         % standard LaTeX graphics tool when including figure files
\usepackage{amsfonts}
\usepackage{amssymb}
\usepackage{eucal}
\usepackage[latin1]{inputenc}
\usepackage[all]{xy}

\newtheorem{thm}{Theorem}[section]
\newtheorem{cor}[thm]{Corollary}
\newtheorem{lemma}[thm]{Lemma}

\newtheorem{prop}[thm]{Proposition}
\theoremstyle{definition}
\newtheorem{defn}[thm]{Definition}
\theoremstyle{remark}

\newtheorem{remark}[thm]{Remark}

\renewcommand{\a}{\alpha}
\renewcommand{\b}{\beta}

\newcommand{\e}{\epsilon}

\newcommand{\PB}{\left\{\cdot\,,\cdot\right\}}

\newcommand{\Pb}[1]{\left\{\cdot\,,#1\right\}}
\newcommand{\pb}[1]{\left\{#1\right\}}

\newcommand{\lb}[1]{\[#1\]}

\renewcommand{\[}{\left[}
\renewcommand{\]}{\right]}

\newcommand{\cE}{\mathcal E}

\newcommand{\cL}{\mathcal L}

\newcommand{\cP}{\mathcal P}
\newcommand{\cO}{\mathcal O}
\newcommand{\cR}{\mathcal R}
\newcommand{\cS}{\mathcal S}
\newcommand{\fS}{\mathfrak S}

\newcommand{\bbC}{\mathbb C}
\newcommand{\C}{\mathbb C}

\newcommand{\bbI}{\mathbb I}

\newcommand{\bbP}{\mathbb P}

\newcommand{\bbZ}{\mathbb Z}

\newcommand{\G}{\mathbf G}

\newcommand{\Z}{\mathbb Z}

\newcommand{\br}{B}

\newcommand{\Ker}{\mathop{\rm Ker}\nolimits}
\newcommand{\diag}{\mathop{\rm diag}\nolimits}
\newcommand{\Res}{\mathop{\rm Res}\nolimits}
\newcommand{\Tr}{\mathop{\rm Tr}\nolimits}
\renewcommand{\Im}{\mathop{\rm Im}\nolimits}
\newcommand{\leqs}{\leqslant}
\newcommand{\geqs}{\geqslant}

\newcommand{\diff}{{\rm d }}

\newcommand{\Cas}{{\mathcal Cas}}

\newcommand{\Ad}{\mathop{\rm Ad}}

\newcommand{\eig}{\mathop{\rm eig}}
\newcommand{\ir}{\mathop{\rm ir}}
\newcommand{\Mat}{{\mathop{\rm Mat}}}
\newcommand{\id}{\mathop{\rm id}}
\newcommand{\Pic}{\mathop{\rm Pic}}
\newcommand{\Gal}{\mathop{\rm Gal}}
\newcommand{\Coker}{\mathop{\rm Coker}}
\newcommand{\Div}{\mathop{\rm Div}}
\newcommand{\Spec}{\mathop{\rm Spec}}
\newcommand{\sm}{\mathop{\rm sm}}
\newcommand{\spl}{\mathop{\rm spl}}
\renewcommand{\geq}{\geqs}
\renewcommand{\leq}{\leqs}
\renewcommand{\div}{\mathop{\rm div}}
\newcommand{\et}{{\mathop{\rm et}}}

\newcommand{\Hom}{{\mathop{\rm Hom}}}

\newcommand{\ue}{\underline{e}}
\newcommand{\uz}{\underline{0}}
\newcommand{\uo}{\underline{\omega}}

\newcommand{\modp}[1]{[#1]}

\newif\ifprivate
\privatefalse

 \numberwithin{equation}{section}

\def\???{\ifprivate {\bf {???}} \marginpar{{\Huge {\bf ?}}}\else \fi}
\numberwithin{equation}{section}

\begin{document}  

%\nocite{*} 

\parskip 4pt
\baselineskip 16pt

%%%%%%%%%%%%%%%%%%%%%%%%%%%%%%%%%
%%%%%%%%%%%%%%%%%%%%%%%%%%%%%%%%%

\title[ACI systems and curves with automorphisms]
{Algebraic integrable systems related to spectral curves with automorphisms}

\author[Rei Inoue]{Rei Inoue}
\address{Rei Inoue, Department of Mathematics and Informatics,
Faculty of Science, Chiba University, Chiba 263-8522, Japan}
\email{reiiy@math.s.chiba-u.ac.jp}

\author[Pol Vanhaecke]{Pol Vanhaecke}
\address{Pol Vanhaecke, Laboratoire de Math\'ematiques et Applications, 
UMR 7348 du CNRS, Universit\'e de Poitiers, 
 86962 Futuroscope Chasseneuil
Cedex, France}
\email{pol.vanhaecke@math.univ-poitiers.fr}

\author[Takao Yamazaki]{Takao Yamazaki} 
\address{Takao Yamazaki, Mathematical Institute, Tohoku University,
  Aoba, Sendai 980-8578, Japan}
\email{ytakao@math.tohoku.ac.jp}

\thanks{The first author is supported by JSPS KAKENHI Grant (22740111),
the third author is supported by JSPS KAKENHI Grant (22684001, 24654001).}

%\date{\today}


\begin{abstract}

We apply a reduction to the Beauville systems to obtain a family of new algebraic completely integrable systems,
related to curves with a cyclic automorphism.

\end{abstract}

\maketitle

%\tableofcontents

%%%%%%%%%%%%%%%%%%%%%%
\section{Introduction}
%%%%%%%%%%%%%%%%%%%%%%
%
Algebraic completely integrable systems (aci systems) occupy a distinguished place among the class of (complex)
integrable systems \cite{adlermoerbekevanhaecke2004,V}. The first example is due to Euler, who shows that the
spinning top which now bears his name can be integrated in terms of elliptic functions. Another well-known
classical example is Kowalevski's top, which is the first example of an integrable system which is integrated in
terms of hyperelliptic theta functions (of genus two). The discovery in the seventies that the KdV equation can
also be integrated in terms of hyperelliptic theta functions (of any genus), revived the interest in integrable
systems. Upon revisiting the classical examples, and several newly constructed ones, Adler and van Moerbeke coined
the term \emph{algebraic complete integrability}, unveiling the (algebro-) geometrical origin and meaning of their
integrability in terms of theta functions: the generic fiber of the momentum map (the generic iso-level set of the
constants of motion) is an affine part of an Abelian variety and the integrable vector fields are translation
invariant on these Abelian varieties.  This new point of view has been the starting point for a rich interaction
between algebraic geometry and integrability.  Many new aci systems have been discovered since then
\cite{avm_so4,adlermoerbekevanhaecke2004,hitchin,B,Mumford}.  We shall recall Beauville's system \cite{B} in \S
\ref{sect:aci-beauville}.  In this system, the generic fiber is the complement of the theta divisor in the Jacobian
variety of a (compact) Riemann surface called the \emph{spectral curve}.  In the present paper, we restrict his
system to the subspace for which spectral curves have certain automorphism, and apply reduction to obtain a new
integrable system.  Below we describe our system in more detail.

Let $p$ be a prime number and set $d=pd'$ for some integer $d'>0$.
Let
$M_{p, d}^{\Delta_{\uo}}$
to be the space of $p \times p$ matrices
whose $(i, j)$-entry is a polynomial of the form
$\ell_{ij}(x)=\sum_{k=0}^d \ell_{ij}^k x^k \in \C[x]$ 
such that $\ell_{ij}^k = 0$ unless $i-j \equiv k \pmod p$.
For $L(x) \in M_{p, d}^{\Delta_{\uo}}$,
its characteristic polynomial $\det(y \bbI_p - L(x))$
can be written as $Q(x^p, y)$
for some $Q(x, y) \in \C[x, y]$.
The assignment $L(x) \mapsto Q(x, y)$ defines
a canonical map
$$ \chi_{\uo} : M_{p, d}^{\Delta_{\uo}} \to \C[x, y]. $$
The space $M_{p, d}^{\Delta_{\uo}}$
is stable under the conjugate action
of the centralizer $G_{\Delta_{\uo}}$ of 
the class $\Delta_{\uo}$ of
$\diag(1, e^{2 \pi i/p}, e^{4 \pi i/p}, \dots, e^{2(p-1) \pi i/p})$ 
in $PGL_p(\C)$
(which is an extension of $\Z/p\Z$ by $(\C^*)^{p-1}$;
see Lemma \ref{lem:gdelta} (2) for details),
and $G_{\Delta_{\uo}}$ acts freely on 
$M_{p, d, \ir}^{\Delta_{\uo}} := \{ L(x) \in M_{p, d}^{\Delta_{\uo}} ~|~
\chi_{\uo}(L(x))$ is irreducible$\}$.
Moreover, the map $\chi_{\uo}$ is
equivariant under this action,
i.e.
we have $\chi_{\uo}(g L(x) g^{-1}) = \chi_{\uo}(L(x))$
for any $L(x) \in  M_{p, d}^{\Delta_{\uo}}$ 
and $g \in G_{\Delta_{\uo}}$.
Hence $\chi_{\uo}$ induces a map
$$ \tilde{\chi}_{\uo} : M_{p, d, \ir}^{\Delta_{\uo}}/G_{\Delta_{\uo}} \to \C[x, y]. $$ In general, the fiber
$\tilde{\chi}_{\uo}^{-1}(Q)$ of $\tilde{\chi}_{\uo}$ over $Q \in \C[x, y]$ is not connected.  We will investigate
the structure of the set of connected components.  Using Beauville's result mentioned above, for generic $Q(x, y)
\in \C[x, y]$ of the form $Q(x, y) = y^{p} + s_1(x)y^{p-1} + \dots + s_p(x)$ with $s_i(x) \in \C[x], ~\deg s_i(x)
\leq d'i$, each connected component of $\tilde{\chi}_{\uo}^{-1}(Q)$ is seen to be isomorphic to an affine open
subset of the Jacobian variety of the Riemann surface defined by the equation $Q(x, y)=0$ (see Theorem
\ref{thm:each-conn-comp-jac}).  
We then combine the methods of 
Poisson-Dirac reduction and Poisson reduction
to construct (several) Poisson structures
$\PB$ on $M_{p, d, \ir}^{\Delta_{\uo}}/G_{\Delta_{\uo}}$.
%
%We then apply the method of Poisson-Dirac reduction to construct (several) Poisson
%structures $\PB$ on $M_{p, d, \ir}^{\Delta_{\uo}}$, which lead upon using Poisson reduction to a family of
%compatible Poisson structures on $M_{p, d, \ir}^{\Delta_{\uo}}/G_{\Delta_{\uo}}$.  
%
%
As Hamiltonian functions, we
take (linear combinations of) regular functions which send $[L(x)] \in M_{p, d,
  \ir}^{\Delta_{\uo}}/G_{\Delta_{\uo}}$ to the coefficient of $x^i y^j$ in $\tilde{\chi}_{\ue}([L(x)])$ with $i, j
\in \Z_{\geq 0}$.  Their Hamiltonian vector fields are shown to be translation invariant.  Therefore we arrive at
our main result (see Theorem \ref{thm:main}):
\begin{equation}\label{eq:main-result}
\text{The triple $(M^{\Delta_{\uo}}_{p, d, \ir}/G_{\Delta_{\uo}}, \PB, \tilde{\chi}_{\uo})$
is an aci system.}
\end{equation}
(See Definition \ref{def:aci} below for a precise definition of an aci system.)  When $p=2$, our system is very
similar to (but not precisely the same as) one of the two hyperelliptic Prym system introduced in \cite{FV}.

Actually, we shall construct a family of aci systems
parameterized by $\ue \in \cE$ 
where $\cE$ is a certain subset of $(\Z/p\Z)^p$
(see \S \ref{sect:p-tor}).
The above system is obtained as a
special case $\ue = \uo := (0, 1, \dots, p-1) \in \cE$.
Suppose we are given general $\ue = (e_1, \dots, e_p) \in \cE$.
Let $G_{\Delta_{\ue}}$ be the centralizer of
the class $\Delta_{\ue}$ of
$\diag(e^{2 \pi e_1 i/p}, \dots, e^{2 \pi e_p i/p})$
in $PGL_p(\C)$.
We shall construct a certain subspace
$M_{p, d}^{\Delta_{\ue}}$ 
of the space of $p \times p$ matrices
whose entries are polynomial of degree $\leq d$
(see \S \ref{sec:fixed_points}),
which is stable under the action of $G_{\Delta_{\ue}}$ by conjugation.
Then everything explained in the previous paragraph
will be carried out for general $\ue$
and we will prove \eqref{eq:main-result}
with $\uo$ is replaced by $\ue$.
However, for general $\ue$
the description of $M_{p, d}^{\Delta_{\ue}}$ 
and $G_{\Delta_{\ue}}$ will be more involved
(see Lemma \ref{lem:mdelta}).

The structure of the paper is as follows. We study in Section 2 the relation between the Jacobians of two curves
which are linked by a ramified cyclic covering of prime order. In Section 3 we introduce a space of polynomial
matrices of size $p$ and study its automorphisms of order $p$, with particular attention to the fixed point set of
such an automorphism. In section 4, both the space and its fixed point set are related, to the corresponding
spectral curves, upon using the momentum map and the results of Section 2. We recall Beauville's result and we use
it to describe the fibers of the aci systems under construction. We deal with the Hamiltonian structure of the
space of polynomial matrices, its fixed point sets and their quotients (by the adjoint action) in Section 5. In
particular we obtain a multi-Hamiltonian structure of our newly constructed phase spaces. 
The algebraic integrability of our system
is proven in Section 6. In the final Section 7 we use \'eale cohomology
to reduce one of the conditions in the main theorem of Section 2.

%%%%%%%%%%%%%%%%%%%%%%%%%%%%%%%%%%%%%%%%%%%%%%%%%%%%%%%%%
\section{Fixed point sets for automorphisms on Jacobians}
%%%%%%%%%%%%%%%%%%%%%%%%%%%%%%%%%%%%%%%%%%%%%%%%%%%%%%%%%
\label{sect2}

For a smooth projective irreducible curve $C$ over $\bbC$, we write $J(C)$ for the Jacobian variety of $C$.  An
automorphism of $C$ leads to an automorphism of $J(C)$. In the present section, we study the fixed point set of the
latter automorphism, in particular we determine the structure and the number of its connected components.

\begin{thm}\label{thm:conn-comp}
Let $C$ and $C'$ be smooth projective irreducible curves over~$\bbC$, and let $f: C \to C'$ be a finite morphism.
Suppose that the corresponding extension $\bbC(C)/\bbC(C')$ of function fields is a Galois extension of prime
degree~$p$.  We denote by $\br\subset C'$ the set of branch points of $f$, and let $N := |\br|$.  Let
$T:=\Gal(\bbC(C)/\bbC(C'))$ be the Galois group of $\bbC(C)/\bbC(C')$.  We suppose that the following two
conditions are satisfied:
\begin{enumerate}
  \item The pull-back $f^* : J(C') \to J(C)$ is injective;
  \item $N>0$.
\end{enumerate}
Then the cokernel of
$$f^* : J(C') \to J(C)^T :=
\{ a \in J(C) ~|~ \tau^*(a)=a ~\text{for all}~ \tau \in T \}
$$ 
is isomorphic to $(\Z/p\Z)^{N-2}$.
\end{thm}

Throughout this section, we keep the notation introduced in Theorem \ref{thm:conn-comp}, in particular $p$ denotes
a prime number and $T\cong \Z/p\Z$ denotes the Galois group of $\bbC(C)/\bbC(C')$. We assume neither (1) nor (2)
until \S \ref{sect:pf-conn-comp}.

\begin{remark}\ 
\begin{enumerate}
\item
The assumptions (1) and (2) of the theorem are redundant because, as we show in the appendix (see Theorem
\ref{thm:inj-jac}), conditions (1) and (2) are {\it equivalent}.  Unfortunately, the proof of the latter
equivalence uses \'etale cohomology.  On the other hand, in the application to integrable systems, both (1) and (2)
are easily verified (see Proposition \ref{prop:inj} and \eqref{eq:no-branch-pt}), thereby the use of \'etale
cohomology can be avoided to establish the main results of this paper.
\item Under the conditions of the theorem, it cannot happen that $N=1$. 
%More generally, a finite abelian extension of
%function fields over $\C$ which ramifies at one point at most must be unramified everywhere.
% (because $H^1(C', \Z/n\Z) \cong H^1(C' \setminus \{ x' \}, \Z/n\Z)$ 
%for any $x' \in C'$ and $n \geqs2$).
\item The isomorphism $\Coker(f^* : J(C') \to J(C)^T) \cong (\Z/p\Z)^{N-2}$ can be explicitly described (see \S
  \ref{sect:isom-exp}), but we will not need this result.
\end{enumerate}
\end{remark}

\subsection{Lemmas on Galois cohomology}\label{sect:convention}
In the rest of this section, we use the following conventions.
For a $T$-module $A$ (i.e. an abelian group on which $T$ acts linearly) we write $H^*(T, A)$ for the group
cohomology of $T$ with values in $A$, so that $H^0(T, A)=A^T:=\{ a \in A ~|~ \tau^*(a)=a ~\text{for all}~ \tau \in
T \}$.  We use the standard notation $\mu_p := \{ \zeta \in \bbC^* ~|~ \zeta^p=1 \}$.  We regard $\bbC^*$ as a
trivial $T$-module.

\begin{lemma}\label{lem:gal-1}
We have
$$
  H^q(T, \bbC^*) \cong\left\{ 
  \begin{array}{ll}
    \bbC^* & ~\text{if}~q=0\;,\\
    \mu_p & ~\text{if}~q \equiv 1 \pmod 2\;,\\
    0 & ~\text{if}~q \equiv 0 \pmod 2,~ q>0\;.
  \end{array}
  \right.
$$
\end{lemma}
\begin{proof}
Choose a generator $\tau$ of $T$.  Recall that for any $T$-module $A$ the cohomology $H^*(T, A)$ can be computed as
a cohomology of the complex (see \cite[Chapter VIII, \S 4]{Serre2})
$$
   A \overset{1-\tau}{\longrightarrow}
   A \overset{D}{\longrightarrow}
   A \overset{1-\tau}{\longrightarrow}
   A \overset{D}{\longrightarrow}
   \dots.
$$
where $D=1+\tau+\dots+\tau^{p-1}$.  If the $T$-module structure on $A$ is trivial, then we have $(1-\tau)(a)=0$ and
$D(a)=pa$ for any $a \in A$. For $A=\bbC^*$ (and upon using multiplicative notation) this leads to the announced
result.
\end{proof}

\begin{lemma}\label{lem:gal-2}
We have
$$
  H^q(T, \bbC(C)^*) \cong\left\{ 
  \begin{array}{ll}
    \bbC(C')^* & ~\text{if}~q =0\;,\\
    0 & ~\text{if}~q>0\;.
  \end{array}
  \right.
$$
\end{lemma}
\begin{proof}
The result for $q=0$ is obvious.  The vanishing for $q>0$ is reduced to the cases $q=1, 2$, because the group
cohomology (in degree $q \geqs 1$) of a cyclic group depends only on the parity of $q$ (see ibid.).  We have
$H^1(T, \bbC(C)^*)=0$ by Hilbert's Theorem 90 (see \cite[Chapter X, Proposition 2]{Serre2}).  As for $H^2(T,
\bbC(C)^*)$, first we note that this group is isomorphic to the subgroup $\mathop{\rm Br}(\bbC(C)/\bbC(C'))$ of the
Brauer group $\mathop{\rm Br}(\bbC(C'))$ of $\bbC(C')$ consisting of all elements split by $\bbC(C)$ (see
\cite[Chapter X, Corollary to Proposition 6]{Serre2}), and then we apply Tsen's theorem to get $\mathop{\rm
  Br}(\bbC(C'))=0$ (see \cite[Chapter X, \S 7]{Serre2}).  The lemma is proved.
\end{proof}

\begin{lemma}\label{lem:gal-3}
We have
\begin{align*}
  &H^q(T, \bbC(C)^*/\bbC^*) \cong\left\{ 
    \begin{array}{ll}
      0 & ~\text{if}~q \equiv 1 \pmod 2\;,\\
      \mu_p & ~\text{if}~q \equiv 0 \pmod 2,\ q>0\;,
    \end{array}
  \right.\\
  &\Coker\Big( \bbC(C')^*/\bbC^* \to (\bbC(C)^*/\bbC^*)^T \Big)\cong \mu_p\;.
\end{align*}
\end{lemma}
\begin{proof}
We consider the following short exact sequence of $T$-modules:
$$ 
  0 \to \bbC^* \to \bbC(C)^* \to \bbC(C)^*/\bbC^* \to 0\;. 
$$
The long exact sequence derived from this, together with the previous lemmas, completes the proof.
\end{proof}

\subsection{Lemmas on Picard and divisor groups}
For any irreducible smooth projective curve $X$ over $\C$, we write $\Pic(X)$ and $\Div(X)$ for the Picard group
and divisor group of $X$ respectively.  Recall that $J(X)$ is identified with the kernel of the degree map $\deg :
\Pic(X) \to \Z$.

\begin{lemma}\label{lem:pic}
We have an isomorphism
$$ 
  \Ker(f^* : J(C') \to J(C)) \cong \Ker(f^* : \Pic(C') \to \Pic(C))
$$
and an exact sequence
$$ 
  0 \to \Coker(J(C') \overset{f^*}{\to} J(C)^T)\to\Coker(\Pic(C') \overset{f^*}{\to} \Pic(C)^T)\overset{(*)}{\to} \Z/p\Z\;. 
$$
Moreover, if $N>0$, then $(*)$ is surjective.
\end{lemma}
\begin{proof}
The first statement is obtained by applying the snake lemma to the following commutative diagram with exact rows
$$
  \begin{matrix}
    0 \to & J(C') & \to & \Pic(C') & \overset{\deg}{\to} & \Z &\to 0\\[1mm]
    & \downarrow^{f^*} &  & \downarrow^{f^*} &  & \downarrow^{p} &\\[1mm]
    0 \to & J(C)^T & \to & \Pic(C)^T & \overset{\deg}{\to} & \Z. &
  \end{matrix}
$$
Suppose that $N>0$.  Let $x \in C$ be a ramification point of $f: C \to C'$.  Then its class $[x] \in \Pic(C)$ is
fixed by $T$ and $\deg([x])=1.$ It follows that $\deg : \Pic(C)^T \to \Z$ is surjective, hence so is $(*)$.
\end{proof}

\begin{lemma}\label{lem:div}
We have
$$ 
  \Coker \Big( f^* : \Div(C') \to \Div(C)^T \Big) \cong (\Z/p\Z)^N\;. 
$$
\end{lemma}
\begin{proof}
For each $x' \in C'$, define a $T$-submodule $D_{x'}$ of $\Div(C)$ by $D_{x'} := \oplus_{x \in f^{-1}(x')} \Z x$.
There is a direct sum decomposition $\Div(C) = \oplus_{x' \in C'} D_{x'}$ as $T$-modules.  Thus we have $\Div(C)^T
= \oplus_{x' \in C'} D_{x'}^T$ and
$$ 
  \Coker \Big( f^* : \Div(C') \to \Div(C)^T \Big) \cong\bigoplus_{x' \in C'} \Coker(f^* : \Z x' \to D_{x'}^T)\;.
$$
If $x'$ is a branch point, then the cokernel of $\Z x' \to D_{x'}^T (=D_{x'})$ is isomorphic to $\Z/p\Z$.  If $x'$
is not a branch point, then $\Z x' \to D_{x'}^T$ is an isomorphism.  The lemma follows.
\end{proof}

\subsection{Proof of Theorem \ref{thm:conn-comp}}\label{sect:pf-conn-comp}
We consider the following commutative diagram with exact rows, 
in which vertical maps are induced by $f:C\to C'$:
$$
\begin{matrix}
  0 \to & \bbC(C')^*/\bbC^* & \to & \Div(C') & \to & \Pic(C') & \to 0\\[1mm]
  & \downarrow^{\alpha} & & \downarrow^{\beta} &  & \downarrow^{\gamma} & \\[1mm]
  0 \to & (\bbC(C)^*/\bbC^*)^T & \to & \Div(C)^T &\overset{\delta}{\to} & \Pic(C)^T\;. &
\end{matrix}
$$
By the last part of Lemma \ref{lem:gal-3}, $\Coker(\alpha)$ is a cyclic group of order $p$.  By Lemma
\ref{lem:div}, we have $\Coker(\beta) \cong (\Z/p\Z)^N$.  By the first part of Lemma \ref{lem:pic} and assumption
(1), $\gamma$ is injective.  Lemma \ref{lem:gal-3} (applied to $q=1$) shows that $\delta$ is surjective.  Therefore
we get $\Coker(\gamma) \cong (\Z/p\Z)^{N-1}$.  Now the theorem follows from the second part of Lemma \ref{lem:pic}
and assumption (2).  
\qed

\subsection{Explicit description of the isomorphism}\label{sect:isom-exp}
Recall that $\br\subset C'$ is the set of branch points of $f$ so that $f$ restricts to a bijection
$f|_{f^{-1}(\br)} : f^{-1}(\br) \to \br$.  We write $\tilde{\beta}$ for the composition of maps
$$
  \bigoplus_{x' \in \br} \Z x' \cong\bigoplus_{x \in f^{-1}(\br)} \Z x \subset \Div(C) \to \Pic(C)\;,
$$
where the first map is induced by the inverse of $f|_{f^{-1}(\br)}$.

By Kummer theory, there exists a rational function $g \in \C(C)$ such that $g'\circ f= g^p$ for some function
$g'\in \C(C')$ and such that $\C(C)$ is generated by $g$ over $\C(C')$.  Since $f$ is unramified outside $\br$, the
divisor $\div(g') \in \Div(C')$ of $g'$ can be written as $\div(g')= D_1 + p D_2$ where $D_1$ (resp. $D_2$) is a
divisor on $C'$ supported on $\br$ (resp. on $C' \setminus \br$).  The proof of Theorem \ref{thm:conn-comp} shows
that there is an exact sequence
$$ 
  0 \to \Z/p\Z \overset{\alpha}{\to}\bigoplus_{x' \in \br} (\Z/p\Z)x'\overset{\beta}{\to} \Coker(\Pic(C') \overset{f^*}{\to} \Pic(C))\to 0\;,
$$
where $\alpha$ is defined by $\alpha(n):=n D_1 \pmod p$, and $\beta$ is induced by $\tilde{\beta}$.  Restricting to
the degree zero part, one obtains a description of the isomorphism given in Theorem \ref{thm:conn-comp}.

%%%%%%%%%%%%%%%%%%%%%%%%%%%%%%%%%%%%%%%%%%%%%%%%%%%%%%%%%%%%%%%%
\section{$PGL_p(\C)$ action on the space of polynomial matrices}
%%%%%%%%%%%%%%%%%%%%%%%%%%%%%%%%%%%%%%%%%%%%%%%%%%%%%%%%%%%%%%%%

\subsection{Space of polynomial matrices}
Let $p \geqs 2,~ d \geqs 1$ be arbitrary integers.  We use the standard notation $\Mat_p(\bbC)$ for the algebra of
all $p \times p$ complex matrices, $GL_p(\bbC)$ for the group of invertible elements of $\Mat_p(\bbC)$ and
$PGL_p(\bbC)$ for the quotient group $GL_p(\bbC)/\{ c \bbI_p ~|~ c \in \bbC^* \}$.

We denote by $M=M_{p, d}$ the set of all $p \times p$ matrices whose entries are polynomials of degree $\leqs d$ in
$x$:
\begin{equation}\label{eq:phase}
  M = \{ L(x) = (\ell_{ij}(x))_{i, j=1, \dots, p}~|~ \ell_{ij}(x) \in \bbC[x], ~ \deg(\ell_{ij}(x)) \leqs d \}\;.
\end{equation}
For $g \in PGL_p(\bbC)$ and $L(x) \in M$, we define $\Ad(g)(L(x)) := \tilde{g} L(x) \tilde{g}^{-1}$ where
$\tilde{g} \in GL_p(\bbC)$ is any representative of $g$.  The class of $L(x) \in M$ in the orbit space
$M/PGL_p(\bbC)$ is denoted by $[L(x)]$.  As we will recall in Section \ref{sec:aci}, $M/PGL_p(\C)$ is the phase
space of the Beauville system, and is therefore fundamental in this paper.

For $L(x) \in M$, we define
$$ 
  {\rm Stab}(L(x)) := \{ g \in PGL_p(\bbC) ~|~ \Ad(g)(L(x))=L(x) \}\;.
$$
We will need the following result.

\begin{lemma}\cite[p.215]{B}\label{lem:stab}
Let $L(x) \in M$.  If the characteristic polynomial $\det(y \bbI_p - L(x)) \in \C[x, y]$ of $L(x)$ is irreducible,
then ${\rm Stab}(L(x)) = \{ 1 \}$.
\end{lemma}

\subsection{The automorphism $\tau$}
We define on $M$ an automorphism $\tau$ of order~$p$~by
\begin{equation}\label{eq:tau}
  \tau : M \to M\;, \qquad  \tau(L(x)):=L(\zeta x)\;,
\end{equation}
where $\zeta:=e^{2\pi i/p}$.  The action of $PGL_p(\bbC)$ on $M$ commutes with $\tau$; namely, we have
$\tau(\Ad(g)(L(x)))=\Ad(g)(\tau(L(x)))$ for any $g \in PGL_p(\bbC)$ and $L(x) \in M$.  Therefore $\tau$ induces a
map
$$ 
  M/PGL_p(\C) \to M/PGL_p(\C)\;,\qquad [L(x)] \mapsto [\tau(L(x))]
$$
which, by abuse of notation, is denoted by the same latter $\tau$.  We define
\begin{align*}
  M' &:= \{ L(x) \in M ~|~ [\tau(L(x))]=[L(x)] ~ \}\;,\\
  M_{\ir}' &:= \{ L(x) \in M' ~|~ \det(y \bbI_r - L(x))\text{ is irreducible } \}\;.
\end{align*}
In view of Lemma \ref{lem:stab}, for each $L(x) \in M_{\ir}'$ there exists a unique $g_{L(x)} \in PGL_p(\bbC)$ such
that $\tau(L(x)) = \Ad(g_{L(x)})(L(x))$.  Since $\tau$ is an automorphism of order $p$, $g_{L(x)}$ must belong to
the closed subset
\begin{equation}\label{eq:cR}
  \cR := \{ \Delta \in PGL_p(\bbC) ~|~ \Delta^p = 1 \}
\end{equation}
of $PGL_p(\bbC)$.  We have defined a map
\begin{equation}\label{eq:map-m-cr}
  M_{\ir}' \to \cR, \qquad L(x) \mapsto g_{L(x)} 
\end{equation}
characterized by $\tau(L(x)) = \Ad(g_{L(x)})(L(x))$.  In the next subsection, we study the structure of $\cR$.

\subsection{$p$-torsion elements in $PGL_p(\bbC)$}\label{sect:p-tor}
We define an equivalence relation on the $p$-fold product $(\Z/p\Z)^p$ of $\Z/p\Z$ as follows: Two elements $(e_1,
\dots, e_p)$, $(e_1', \dots, e_p') \in (\Z/p\Z)^p$ are equivalent if and only if there exist $c \in \Z/p\Z$ and
$\sigma \in \fS_p$ such that $e_i = c + e_{\sigma(i)}'$ (in $\Z/p\Z)$ for all $i=1, \dots, p$.  (Here $\fS_p$
denotes the permutation group on $p$ letters.)  We choose a system of representatives $\cE \subset (\Z/p\Z)^p$ for
this equivalence relation.  For simplicity, we assume that the two elements $\uz := (0, 0, \dots, 0)$ and $\uo :=
(0, 1, 2, \dots, p-1)$ of $(\Z/p\Z)^p$ belong to $\cE$.  (In Lemma \ref{lem:rep} below, we give an explicit
construction of a system of representatives $\cE$.)
% For $\ue = (e_1, \dots, e_p) \in (\Z/p\Z)^p$ we denote by $[\ue]=[e_1, \dots, e_p] \in \cE$ the (unique) element
%which is equivalent to $\ue$.  Denote by $\cE$ the quotient space of $(\Z/p\Z)^p$ by this equivalence relation.
%The class of $(e_1, \dots, e_p) \in (\Z/p\Z)^p$ is denoted by $[e_1, \dots, e_p] \in \cE$.

Let $\ue =(e_1, \dots, e_p) \in (\Z/p\Z)^p$.  We define $\Delta_{\ue}$ to be the element of $PGL_p(\bbC)$
represented by the diagonal matrix $\diag(\zeta^{e_1}, \zeta^{e_2}, \dots, \zeta^{\e_p}) \in GL_p(\bbC)$, where
$\zeta:=e^{2\pi i/p}$.  We also define
\begin{equation}
  \cR_{\ue} := \{ g \Delta_{\ue} g^{-1} \in \cR ~|~ g \in PGL_p(\bbC) \}\;.
\end{equation}
\begin{lemma}\label{lem:cR}
The closed subset $\cR$ of $PGL_p(\bbC)$ defined in \eqref{eq:cR} admits a disjoint union decomposition
$$ 
  \cR = \bigsqcup_{\ue \in \cE} \cR_{\ue}\;. 
$$
Moreover, $\cR_{\ue}$ is a connected component of $\cR$ for each $\ue \in \cE$.  In particular, the number of
connected components in $\cR$ is the same as the cardinality $|\cE|$ of $\cE$ (cf. Lemma \ref{lem:rep}).
\end{lemma}
\begin{proof}
Similarly to $\cE$, we define $\cS$ to be the quotient space of the $p$-fold product $(\bbC^*)^p$ of $\bbC^*$ (with
the quotient topology) by the following equivalence relation: two elements $(a_1, \dots, a_p), (b_1, \dots, b_p)
\in (\bbC^*)^p$ are equivalent if and only if $(a_1, \dots, a_p)=(cb_{\sigma(1)}, \dots, cb_{\sigma(p)})$ for some
$c \in \bbC^*, ~\sigma \in \fS_p$.
%The class of $(a_1, \dots, a_p) \in (\bbC^*)^p$
%is denoted by $[a_1, \dots, a_p] \in \cS$.
%%The map $(\Z/p\Z)^p \to (\bbC^*)^p, ~
%(e_1, \dots, e_p) \mapsto (\zeta^{e_1}, \dots, \zeta^{e_p})$
%induces an injective map
%$\cE \to \cS,~
%[e_1, \dots, e_p] \mapsto [\zeta^{e_1}, \dots, \zeta^{e_p}]$.
%We regard $\cE$ as a subset of $\cS$ by this inclusion.
%
By associating to an element of $PGL_p(\bbC)$ the list of its $p$ eigenvalues (with multiplicities) we obtain a
continuous map $\eig : PGL_p(\bbC) \to \cS$.  For any $\ue=(e_1, \dots, e_p) \in \cE$, the image of $\cR_{\ue}$ by
$\eig$ consists of the single element $\eig(\Delta_{\ue}) \in \cS$ represented by $(\zeta^{e_1}, \zeta^{e_2},
\dots, \zeta^{\e_p}) \in (\bbC^*)^p$, i.e. $\cR_{\ue} \subset \eig^{-1}(\eig(\Delta_{\ue}))$.

Now we claim that for any $\Delta \in \cR$ there is a unique $\ue \in \cE$ such that $\Delta$ is conjugate to
$\Delta_{\ue}$ in $PGL_p(\bbC)$ (so that $\eig(\Delta)=\eig(\Delta_{\ue})$).  We choose a $\tilde{\Delta} \in
GL_p(\bbC)$ representing $\Delta$.  Then we have $\tilde{\Delta}^p = c \bbI_p$ for some $c \in \bbC^*$.  We may
assume $c=1$ by replacing $\tilde{\Delta}$ by $c^{-1/p} \tilde{\Delta}$.  Then the minimal polynomial of
$\tilde{\Delta}$ is a divisor of $x^p-1$, which has no multiple root.  It follows that $\tilde{\Delta}$ is a
diagonalizable matrix all of whose eigenvalues are $p$-th roots of unity, so that they can be written as
$(\zeta^{e_1},\dots,\zeta^{e_p})$, which is a representative of 
$\eig(\Delta_{\ue})$ for a unique $\ue\in\cE$.

Therefore the map $\eig : PGL_p(\bbC) \to \cS$
restricts to a continuous map
$\eig|_{\cR} : \cR \to \cS' := \{ \eig(\Delta_{\ue}) ~|~ \ue \in \cE \}$.
Let $\ue \in \cE$.
It follows from the above claim that
$\cR_{\ue} = \eig|_{\cR}^{-1}(\eig(\Delta_{\ue}))$.
Since $\cS'$ is finite and hence discrete in $\cS$,
we find that $\cR_{\ue}$ is open and closed in $\cR$, so it is the disjoint union of connected components of $\cR$.
As $\cR_{\ue}$ is the image of a continuous map
$PGL_p(\bbC) \to PGL_p(\bbC)$ defined by
$g \mapsto g \Delta_{\ue} g^{-1}$,
$\cR_{\ue}$ is connected.
This completes the proof.
\end{proof}
By this lemma, we obtain a map
\begin{equation}\label{eq:map-cr-ce}
  c_{\cE} : \cR \twoheadrightarrow \cE
\end{equation}
characterized by the property $c_{\cE}(\cR_{\ue}) = \{ \ue \}$ for all $\ue \in \cE$.  Note that we have
$c_{\cE}^{-1}(\uz) = \{ 1 \}$.  The map
\begin{equation}\label{eq:map-m-ce}
  M_{\ir}' \to \cE\;,
\end{equation}
obtained by composing $c_{\cE}$ with \eqref{eq:map-m-cr}, 
will be used later.

\subsection{Construction of $\cE$}
Let $\cP$ be the set of all $p$-term sequences $(\mu_i)_{i \in \Z/p\Z}$ of non-negative integers (indexed by
$\Z/p\Z$) such that $\sum_{i \in \Z/p\Z} \mu_i = p$.  The cardinality of $\cP$ is $\binom{2p-1}{p-1}$.  We define a
map $s : \cP \to (\Z/p\Z)^p$ by $s((\mu_i)_{i \in \Z/p\Z})=(e_1, \dots, e_p)$, where $e_1, \dots, e_p$ are defined
by the condition
$$
  e_j = i \quad \text{if and only if} \quad 1 + \sum_{k=0}^{i-1} \mu_k \leq j \leq \sum_{k=0}^{i} \mu_k
$$
for $i \in \{ 0, \dots, p-1 \}$.
%
%
%For $(\mu_i)_{i \in \Z/p\Z} \in \cP$,
%we define
%$s((\mu_i)_{i \in \Z/p\Z})=(e_1, \dots, e_p) \in \cE$ by 
%$$
%e_j = i \quad \text{if and only if} \quad
%1 + \sum_{k=0}^{i-1} \mu_k \leq j \leq \sum_{k=0}^{i} \mu_k
%$$
%for $0 \leq i \leq p-1$,
%yielding  a map $s : \cP \to (\Z/p\Z)^p$.
%
%
We define two elements $(\mu_i)_{i \in \Z/p\Z},~ (\mu_i')_{i \in \Z/p\Z} \in \cP$ to be equivalent if and only if
$\mu_i = \mu_{i+c}'$ for some $c \in \Z/p\Z$.  We choose a representative $\cP' \subset \cP$ for this equivalence
relation.  One method to construct such a $\cP'$ is to choose a total ordering on $\cP$ (e.g. the lexicographic
order) and define $\cP'$ to be the set of all elements which are maximal in their equivalence class.

\begin{lemma}\label{lem:rep}
  Set $\cE := s(\cP') \subset (\Z/p\Z)^p$.
  \begin{enumerate}
    \item With respect to the equivalence relation introduced in \S \ref{sect:p-tor}, $\cE$ is a representative of $(\Z/p\Z)^p$.
    \item If $p$ is a prime number, then we have
      $$ |\cE|=|\cP'|= \frac{1}{p}(\binom{2p-1}{p-1}-1) + 1\;. $$
  \end{enumerate}
\end{lemma}
\begin{proof}
For $\ue = (e_1, \dots, e_p) \in (\Z/p\Z)^p$, we define $\mu(\ue) =(\mu_i)_{i \in \Z/p\Z} \in \cP$ by
$$
  \mu_i := |\{ j \in \{ 1, \dots, p \}~|~ e_j = i ~\text{(in $\Z/p\Z$)} \}|\;.
$$
This defines a map $\mu : (\Z/p\Z)^p \to \cP$.  It is easy to check that $\mu \circ s = \id_{\cP}$.  (In
particular, $s$ is injective and $\mu$ is surjective.)  Moreover, for $\ue=(e_1, \dots, e_p), \ue'=(e_1', \dots,
e_p') \in (\Z/p\Z)^p$ one has $\mu(\ue)=\mu(\ue')$ if and only if there exists $\sigma \in \fS_p$ such that $e_j =
e_{\sigma(j)}'$ for all $j=1, \dots, p$.  Finally, for $\mu(e_1, \dots, e_p) = (\mu_i)_{i \in \Z/p\Z}$ and $c \in
\Z/p\Z$, we have $\mu(e_1 + c, \dots, e_p + c)=(\mu_{i+c})_{i \in \Z/p\Z}$.  (1) is a consequence of these
observations.

Suppose now that $p$ is prime.  Then the equivalence class of $(\mu_i)_{i \in \Z/p\Z} \in \cP$ contains either $p$
elements or only one element; the latter occurs only in the case $\mu_i=0$ for all $i$.  This shows (2).
\end{proof}

\subsection{Fixed points}\label{sec:fixed_points}
Let $\Delta \in \cR$.  We define
\begin{align*}
  \sigma_{\Delta}:& M \to M, \quad\sigma_{\Delta}(L(x)):=\Ad(\Delta^{-1})(\tau(L(x)))=\tau(\Ad(\Delta^{-1})(L(x)))\;,\\
  M^{\Delta} &:=\{ L(x) \in M ~|~ \sigma_{\Delta}(L(x))=L(x) \} ~~(\subset M')\;,\\
  M_{\ir}^{\Delta} &:= M^{\Delta} \cap M_{\ir}'\;.
\end{align*}
Thus, $\sigma_\Delta$ is an automorphism of $M$ of order $p$ and $M^\Delta$ is its fixed point locus, which is a
linear subspace of $M$. For any $g \in PGL_p(\C)$, we have that $g \Delta g^{-1} \in \cR$ and we have a linear
isomorphism
\begin{equation}\label{eq:delta-iso}
  M^{\Delta} \cong M^{g \Delta g^{-1}}, \qquad L(x) \mapsto \Ad(g)(L(x))\;.
\end{equation}
Now by Lemma \ref{lem:cR}, there is a unique $\ue \in \cE$ such that $\Delta = g \Delta_{\ue} g^{-1}$ for some $g
\in PGL_p(\bbC)$.  Hence we get an isomorphism
\begin{equation}\label{eq:delta-iso2}
  M^{\Delta_{\ue}} \cong M^{\Delta}
\end{equation}
which restricts to $M^{\Delta_{\ue}}_{\ir} \cong M^{\Delta}_{\ir}$.  In Lemma \ref{lem:mdelta} below, we provide an
explicit description of $M^{\Delta_{\ue}}$ for all $\ue \in (\Z/p\Z)^p$.

To state it, we need some more notation.  For $i, j \in \Z, ~ 1 \leq i, j \leq p$, we use the standard notation
$E_{ij} \in \Mat_p(\bbC)$ for the $p \times p$ matrix whose only non-zero entry is its $(i,j)$-th entry, which is
$1$.  For $\ue=(e_1, \dots, e_p) \in (\Z/p\Z)^p$ and $k \in \Z$, we define a subspace of $\Mat_p(\bbC)$ by
$$ 
  D_{\ue, k}:=\bigoplus_{e_i-e_j \equiv k \atop \pmod p}\bbC E_{ij}\;, 
$$
where $i, j$ run through all $1 \leq i, j \leq p$ such that $e_i-e_j \equiv k \pmod p$.

\begin{lemma}\label{lem:mdelta}
  Let $\ue = (e_1, \dots, e_p)  \in (\Z/p\Z)^p$. Then we have
  \begin{equation}\label{eq:mdelta}
    M^{\Delta_{\ue}}=\bigoplus_{k=0}^{d} D_{\ue, k} x^{k}\;.
  \end{equation}
\end{lemma}
\begin{proof}
By the definition of $\tau$ and $\Delta_{\ue}$, one sees
$$ 
  \tau(E_{ij}x^k)=\zeta^k E_{ij}x^k\;,\qquad\Ad(\Delta_{\ue})(E_{ij}x^k)=\zeta^{e_i-e_j} E_{ij}x^k\;,
$$
from which \eqref{eq:mdelta} follows.
%The second assertion is deduced  from
%$$
%D_{\uz, k} = \Mat_p(\bbC) ~\text{if}~k \equiv 0 \pmod p,
%\quad
%D_{\uz, k} = 0 ~\text{if}~k \not\equiv 0 \pmod p,
%$$
%and $\dim D_{\uo, k}=p$ for all $k \in \Z$.
\end{proof}

We extract a detailed description for the two special cases as a corollary.
\begin{cor}\
  \begin{enumerate}
    \item For $\uz = (0, 0, \dots, 0)$, we have
      \begin{equation*}
        M^{\Delta_{\uz}}=\bigoplus_{k=0}^{\lfloor d/p \rfloor} {\Mat}_p(\bbC)x^{pk}\;,
      \end{equation*}
      where $\lfloor \cdot \rfloor$ denotes the floor function.  In particular, the dimension of $M^{\Delta_{\uz}}$
      is $({\lfloor d/p \rfloor}+1)p^2$.
    \item For $\uo = (0, 1, 2, \dots, p-1)$, we have
      \begin{equation*}
        M^{\Delta_{\uo}}=\bigoplus_{k=0}^d\bigoplus_{1 \leq i, j \leq p \atop i-j \equiv k \pmod p}\bbC \cdot E_{ij}x^k\;.
      \end{equation*}
      In particular, its dimension is $(d+1)p$.
  \end{enumerate}
\end{cor}
\begin{proof}
This follows immediately from the lemma.
\end{proof}

\subsection{A decomposition of $M'_{\ir}$}
Let $\ue \in \cE$.  We define
\begin{align}\label{eq:m-ue}
  &M^{\ue} := \{ L(x) \in M ~|~ \sigma_{\Delta}(L(x))=L(x) ~\text{for some}~ \Delta \in \cR_{\ue} \}\subset M'\;,\\
  \label{eq:m-ue-ir}&M^{\ue}_{\ir} := M^{\ue} \cap M_{\ir}'\;.
\end{align}
They are stable under the action of $PGL_p(\C)$ (see \eqref{eq:delta-iso}).  By Lemma \ref{lem:cR} and the above
definitions, there are disjoint union decompositions
\begin{equation}\label{eq:decomp-mue}
  M_{\ir}' = \bigsqcup_{\ue \in \cE} M_{\ir}^{\ue}\;,\qquad M_{\ir}^{\ue} = \bigsqcup_{\Delta \in \cR_{\ue}} M^{\Delta}_{\ir}\;.
\end{equation}
For each $\alpha \in PGL_p(\C)$, we define
\begin{equation}\label{eq:g-alpha}
  G_{\alpha} := \{ g \in PGL_p(\bbC) ~|~ g \alpha g^{-1} = \alpha \}\;,
\end{equation}
the centralizer of $\alpha$, which is a subgroup of $PGL_p(\bbC)$.
\begin{lemma}\label{lem:m-delta-vs-m-ue}
If $\ue \in \cE$ and $\Delta \in \cR_{\ue}$, then the inclusion $M^{\Delta}_{\ir} \hookrightarrow M^{\ue}_{\ir}$
induces an isomorphism
$$ 
  M^{\Delta}_{\ir}/G_{\Delta} \cong M^{\ue}_{\ir}/PGL_p(\bbC)\;. 
$$
\end{lemma}
\begin{proof}
The surjectivity follows from \eqref{eq:delta-iso2} and \eqref{eq:decomp-mue}.  To show the injectivity, we take
$L_i(x) \in M^{\Delta}_{\ir} ~(i=1, 2)$ and $g \in PGL_p(\bbC)$ satisfying $L_1(x)=\Ad(g)(L_2(x))$.  Then we have
\begin{align*}
  \Ad(g \Delta)(L_2(x))&= \Ad(g)(\Ad(\Delta)(L_2(x)))= \Ad(g)(\tau(L_2(x)))\\
  &= \tau(\Ad(g)(L_2(x)))= \tau(L_1(x))= \Ad(\Delta)(L_1(x))\\
  &= \Ad(\Delta)(\Ad(g)(L_2(x)))= \Ad(\Delta g)(L_2(x)))\;.
\end{align*}
By Lemma \ref{lem:stab}, we get $g \in G_{\Delta}$. This completes the proof.
\end{proof}

For $k \in \Z/p\Z$, we define $D_k^*$ to be the subset of $PGL_p(\bbC)$ consisting of all elements represented by
$d=(d_{ij}) \in GL_p(\bbC)$ such that $d_{ij}=0$ if $i-j \not\equiv k \pmod p$ and that $d_{ij}\not=0
~\text{if}~i-j \equiv k \pmod p$.

\begin{lemma}\
\label{lem:gdelta}
\begin{enumerate}
\item
For $\uz = (0, 0, \dots, 0) \in \cE$, we have $G_{\Delta_{\uz}} = PGL_p(\bbC)$.
\item
For $\uo = (0, 1, 2, \dots, p-1) \in \cE$, we have
\begin{equation*}
  G_{\Delta_{\uo}}=\bigsqcup_{k=0}^{p-1} D_{k}^*\;,
\end{equation*}%
Each of $D_0^*, \dots, D_{p-1}^*$ is a connected component of $G_{\Delta_{\uo}}$, and $D_0^*$ is a subgroup of
index $p$ in $G_{\Delta_{\uo}}$.
\end{enumerate}
\end{lemma}
\begin{proof}
(1) is obvious and (2) follows from $\Ad(\Delta_{\uo})(a_{ij})=(\zeta^{i-j} a_{ij})$ for any $(a_{ij}) \in
  \Mat_p(\bbC)$.
\end{proof}

%%%%%%%%%%%%%%%%%%%%%%%%%%%%%%%%%%%%%%%%%%%%%%%%%
\section{The momentum map and its generic fibers}
%%%%%%%%%%%%%%%%%%%%%%%%%%%%%%%%%%%%%%%%%%%%%%%%%
%
We fix arbitrary integers $p \geq 2, ~d \geq 1$.

\subsection{The level sets}\label{sect:levelsets}
We write $V =V_{p, d} \subset \bbC[x, y]$ for the affine space of all polynomials of the form
$$ 
  y^p + s_1(x) y^{p-1} + \cdots%+ s_i(x)y^{r-i} + \cdots 
  + s_p(x)\;, \quad \deg(s_i(x)) \leqs id ~(i=1, \dots, p)\;.
$$
The characteristic polynomial $\det(y \bbI_p - L(x))$ of any element $L(x)$ of $M=M_{p, d}$ belongs to $V$,
therefore we obtain a map
$$
  \chi : M \to V, \qquad \chi(L(x)) := \det(y \bbI_p - L(x))\;.
$$
Since conjugate matrices have the same characteristic polynomial, $\chi$ induces a map
\begin{equation}\label{eq:chitilde}
  \tilde{\chi} : M/PGL_p(\bbC) \to V\;,\quad\tilde{\chi}([L(x)]) = \chi(L(x))\;.
\end{equation}
As we will recall in Section \ref{sec:aci}, $\tilde{\chi}$ is the momentum map of the Beauville system, so it is
fundamental in this paper.

Let $P=P(x, y) \in V$ and write $P = y^p + s_1(x) y^{p-1} + \cdots + s_p(x)$ with $s_i(x) \in \C[x], ~\deg(s_i(x))
\leq id$.  For $i=1, \dots, p$, we denote by $s_i(\infty) \in \C$ the coefficient of $x^{di}$ in $s_i(x)$, and we put
$P(\infty, y) := \sum_{i=0}^p s_i(\infty) y^{p-i}$.  For $L(x) = \sum_{k=0}^d L_k x^k \in M ~(L_k \in \Mat_p(\C))$, we
write $L(\infty) := L_d \in \Mat_p(\C)$.  Note that $\chi(L)(\infty, y) = \det(y \bbI_p - L(\infty))$ for any $L(x)
\in M$.  For any $P \in V$ we write $M_P$ for the fiber of $\chi$ over $P$,
$$
  M_P := \{ L(x) \in M ~|~ \chi(L(x))=P \}\;.
$$
Then $M_P$ is stable under the action of $PGL_p(\C)$ on $M$.  We are going to describe $M_P/PGL_p(\C)$ in terms of
Jacobian varieties under some assumptions on $P$.

\subsection{Spectral curve}\label{sect:spec-curve}
Fix $P(x, y) \in V$, and let us introduce another polynomial $P'(z,w):=z^{pd} P(z^{-1},z^{-d} w) \in
\mathbb{C}[z,w]$.  We define a projective curve $C_P$ over $\C$ by gluing two affine curves $\Spec \C[x, y]/(P(x,
y))$ and $\Spec \C[z, w]/(P'(z, w))$ by the relation $x=z^{-1}, ~y=z^{-d} w$.  The rational function $x=z^{-1}$ on
$C_P$ defines a finite morphism $\pi_P : C_P \to \bbP^1$ of degree $p$.

We define open subsets of $V$ as follows:
\begin{align*}
  &V_{\ir} := \{ P(x, y) \in V ~|~ C_P ~\text{is integral} ~\}\;,\\
  &V_{\sm} := \{ P(x, y) \in V_{\ir} ~|~ C_P ~\text{is smooth} ~\}\;,\\
  &V_{\spl} := \{ P(x, y) \in V_{\sm} ~|~ \pi_P ~\text{is unramified at}~\infty \in \bbP^1 ~ \}\;.
\end{align*}

\goodbreak

\begin{remark}\ 
  \begin{enumerate}
    \item
      Let $P \in V$.  One has that $P \in V_{\ir}$ if and only if $P$ is irreducible.  If this is the case, Lemma
      \ref{lem:stab} shows that $PGL_p(\C)$ acts freely on $M_P$.
    \item For $P \in V_{\sm}$, one has that $P \in V_{\spl}$ if and only if $P(\infty, y)$ has no multiple root.
  \end{enumerate}
\end{remark}

We will need the following results:

\begin{lemma}\cite[(1.2)]{B}
\label{lem:genus}
Let $P \in V_{\sm}$.
Then the genus of $C_P$ is $g_P := \frac{1}{2}(p-1)(pd-2)$.
\end{lemma}

\begin{lemma}\cite[(1.12)]{B}
If $P \in V_{\sm}$, then $L(c)$ is regular for any $L(x) \in M_P$ and any $c \in \bbP^1$.  (Recall that $A \in
\Mat_p(\C)$ is called regular if its minimal polynomial is of degree $p$.)
\end{lemma}

\subsection{A result of Beauville}\label{sect:bea}
We fix $P \in V_{\sm}$.  For each $n \in \Z$, we set $\Pic^n(C_P):=\{ \cL \in \Pic(C_P) ~|~ \deg(\cL)=n \}$.
Recall that the Jacobian variety $J(C_P) = \Pic^0(C_P)$ has a structure of an abelian variety of dimension $g_P$,
and $\Pic^n(C_P)$ is a torsor under $J(C_P)$ for each $n \in \Z$.  We define the {\it theta divisor} by
$$
  \Theta_P := \{ \cL \in {\Pic}^{g_P-1}(C_P) ~|~H^0(C_P, \cL) \not= 0 \}\;.
$$
The following fundamental result is due to Beauville \cite[Thm. 1.4]{B}.
\begin{thm}[Beauville]\label{thm:bea}
  Let $P \in V_{\sm}$.  Then there is an isomorphism
  $$ M_P/PGL_p(\C) \cong {\Pic}^{g_P-1}(C_P) \setminus \Theta_P\;. $$
\end{thm}

We briefly recall the construction of this isomorphism, following \cite{B}.  Fix $P \in V_{\sm}$ and take an
invertible sheaf $\cL$ of degree $g_P-1$ on $C_P$.  Then $\pi_{P *} \cL$ is a free $\cO_{\bbP^1}$-module of rank
$p$.  Moreover, it is isomorphic to $\cO_{\bbP^1}(-1)^p$ if and only if the class of $\cL$ belongs to
$\Pic^{g_P-1}(C_P) \setminus \Theta_P$.  Suppose this is the case and fix an isomorphism $\alpha : \pi_{P *} \cL
\cong \cO_{\bbP^1}(-1)^p$.  The morphism $\pi_{P *} \cL \to \pi_{P *} \cL \otimes \cO_{\bbP^1}(d)$ given by
``multiplication by $y$'' induces, via $\alpha$, a morphism $\cO_{\bbP^1}(-1)^p \to \cO_{\bbP^1}(d-1)^p$ which is
represented by a matrix $L_{(\cL, \alpha)}(x) \in M$.  The relation $P(x, y)=0$ imposes $L_{(\cL, \alpha)}(x) \in
M_P$.  If $\beta : \pi_{P *} \cL \cong \cO_{\bbP^1}(-1)^p$ is another isomorphism, then $L_{(\cL, \alpha)}(x)$ and
$L_{(\cL, \beta)}(x)$ define the same class in $M_P/PGL_p(\C)$.  The isomorphism in Theorem \ref{thm:bea}.  is
given by this correspondence.

\subsection{Fixed points of $\tau$}
For the rest of this section, we assume that $p$ is a prime number and that $d=d'p$ for some integer $d' \geq 1$.
We apply the results of the previous sections to $M=M_{p, d}$ and also to $M_{p, d'}$, and study their relation.
Thus, to clarify the distinction, we will, for instance, write $M_{p, d, P}$ for what has been written as $M_P$.
We set 
$$ 
  V_{p, d', \spl}' := \{ Q \in V_{p, d', \spl}~|~ Q(0, y)~\text{has no multiple root} \}\;.
$$
Note that for $Q \in V_{p, d', \spl}$ the two conditions $Q(x^p, y) \in V_{p, d, \sm}$ and $Q(x, y) \in V_{p, d',
\spl}'$ are equivalent, and if this is the case then $Q(x^p, y) \in V_{p, d, \spl}$.

Now we fix $Q \in V_{p, d', \spl}'$
and set $P:=Q(x^p, y) \in V_{p, d, \spl}$.
We define morphisms
$f_Q : C_P \to C_Q$
and
$f_{\bbP^1} : \bbP^1 \to \bbP^1$
by $f_Q(x, y) \mapsto (x^p, y)$
and $f_{\bbP^1}(x)=x^p$ so that
we obtain a commutative diagram
$$
\begin{matrix}
C_P & \overset{\pi_P}{\to} & \bbP^1
\\[1mm]
{}^{f_Q} \downarrow \quad 
& &  \quad \downarrow^{f_{\bbP^1}}
\\[1mm]
C_Q & \overset{\pi_Q}{\to} & \bbP^1\;.
\end{matrix}
$$
The branch locus of $f_Q$ is exactly
$\pi_Q^{-1}(\{ 0, \infty \})$,
hence we have
\begin{equation}\label{eq:no-branch-pt}
  \text{for any $Q \in V_{p, d', \spl}'$,~ the number of branch points of $f_Q$ is $2p$\;.}
\end{equation}
(This can be also seen by the Riemann-Hurwitz formula and Lemma \ref{lem:genus}).

The automorphism $\tau$ on $M_{p, d}$ (see \eqref{eq:tau}) restricts to an isomorphism on $M_{p, d, P}$.  Let
$M_{p, d}^{\tau} := \{ L(x) \in M_{p, d} ~|~ \tau(L)=L \}$ be the set of fixed points of $\tau$ on $M_{p, d}$.
Then $M_{p, d, P}^{\tau} := M_{p, d, P} \cap M_{p,d}^{\tau}$ is identified with $M_{p, d', Q}$ under the
correspondence
\begin{equation}\label{eq:p-th-power-map}
  M_{p, d', Q} \cong M_{p, d, P}^{\tau}\;,\qquad L(x) \mapsto L(x^p)\;.
\end{equation}
The composition of this map with the inclusion $M_{p, d, P}^{\tau} \hookrightarrow M_{p, d, P}$ induces the left
vertical map in the following commutative diagram:
\begin{equation}\label{eq:comm1}
  \begin{matrix}
    M_{p, d', Q}/PGL_p(\C)& \overset{\text{Thm. \ref{thm:bea}}}{\cong} &\Pic^{g_Q-1}(C_Q) \setminus \Theta_Q\\[1mm]
    {}_{\text{\eqref{eq:p-th-power-map}}}\downarrow \quad& &\downarrow\\[1mm] 
    M_{p, d, P}/PGL_p(\C)& \underset{\text{Thm. \ref{thm:bea}}}{\cong} &\Pic^{g_P-1}(C_P) \setminus \Theta_P\;,
  \end{matrix}
\end{equation}
where the right vertical map is defined as
$$ 
  \cL \mapsto {f_Q}^*(\cL) \otimes {\pi_P}^*(\cO_{\bbP^1}(p-1))\;. 
$$ 
As the left vertical map is obviously injective, we obtain the following result.
%(which is a special case of Theorem \ref{thm:conn-comp} (1)).

\begin{prop}\label{prop:inj}
  The pull-back map $f_Q^* : J(C_Q) \to J(C_P)$ is injective.
\end{prop}

Next we consider $M_{p, d, P}' := M_{p, d, P} \cap M' (\subset M_{\ir}')$.  It is obvious from the definition that
$M_{p, d, P}^{\tau} \subset M_{p, d, P}'$, but they do not coincide.  Actually, the decomposition
\eqref{eq:decomp-mue} restricts to
\begin{equation}\label{eq:decomp-mp'}
  M_{p, d, P}' = \bigsqcup_{\ue \in \cE} M_{p, d, P}^{\ue}\;,\quad M_{p, d, P}^{\ue} := M_{p, d, P} \cap M_{p, d}^{\ue}\;,
\end{equation}
and for $\uz = (0, \dots, 0) \in \cE$, we have $M_{p, d, P}^{\tau} = M_{p, d, P}^{\uz}$.  For general $\ue \in
\cE$, $M_{p, d, P}^{\ue}$ is nothing other than the inverse image of $\ue$ by the map
\begin{equation}\label{eq:map-mp-ce}
  M_{p, d, P}' \to \cE
\end{equation}
obtained by composing the inclusion $M_{p, d, P}' \hookrightarrow M_{\ir}'$ with \eqref{eq:map-m-ce}.  By
definition, \eqref{eq:map-mp-ce} is a continuous map, but $\cE$ is a finite discrete set.  It follows that $M_{p,
  d, P}^{\ue}$ is open and closed in $M_{p, d, P}'$, that is, a union of (a finite number of) connected components
of $M_{p, d, P}'$.

For each $\ue \in \cE$, we consider the subspace (see \eqref{eq:m-ue})
$$
  M_{p, d, P}^{\Delta_{\ue}}:= M_{p, d, P} \cap M_{p, d}^{\Delta_{\ue}}\;. 
$$
of $M_{p, d, P}$, which is stable under the action of $G_{\Delta_{\ue}}$.  The main result of this subsection is
the following:

\begin{thm}\label{thm:each-conn-comp-jac}
  Let $\ue \in \cE,~ Q \in V_{p, d', \spl}'$ and set $P:=Q(x^p, y) \in V_{p, d, \spl}$.  Then every connected
  component of $M_{p, d, P}^{\Delta_{\ue}}/G_{\Delta_{\ue}}$ is isomorphic to an affine open subset of $J(C_Q)$.
\end{thm}
\begin{proof}
By restricting the isomorphism given in Lemma \ref{lem:m-delta-vs-m-ue}, we obtain an isomorphism
$$
  M_{p, d, P}^{\Delta_{\ue}}/G_{\Delta_{\ue}}\cong M_{p, d, P}^{\ue}/PGL_p(\C)\;.
$$
In view of \eqref{eq:decomp-mp'}, $M_{p, d, P}^{\ue}/PGL_p(\C)$ is a union of some connected components of
$M_{p,d,P}'/PGL_p(\C)$, which is isomorphic to the complement of $\Theta_P$ in
$$
  {\Pic}^{g_P-1}(C_P)^{\tau} := \{ \cL \in {\Pic}^{g_P-1}(C_P) ~|~ \tau^*(\cL)=\cL \}
$$
by Theorem \ref{thm:bea}.  Let $X$ be a connected component of ${\Pic}^{g_P-1}(C_P)^{\tau}$.  Theorem
\ref{thm:conn-comp}, Proposition \ref{prop:inj} and \eqref{eq:no-branch-pt} show that $X$ is isomorphic to $J(C_Q)$
(as algebraic varieties).  As $\Theta_P$ is an ample divisor on ${\Pic}^{g_P-1}(C_P)$, its restriction to $X$ is
again ample, and hence $X \setminus \Theta_P$ is an affine open subvariety of~$X$.  This completes the proof.
\end{proof}

\begin{remark}
Let $\pi_0$ be the set of connected components of ${\Pic}^{g_P-1}(C_P)^{\tau}$, whose cardinality is $p^{2p-2}$ by
Theorem \ref{thm:conn-comp}, Proposition \ref{prop:inj} and \eqref{eq:no-branch-pt}.  There is a map $\pi_0 \to
\cE$ which sends $X \in \pi_0$ to unique $\ue \in \cE$ such that the image of $M_{p, d, P}^{\ue}/PGL_p(\C)$ in
${\Pic}^{g_P-1}(C_P)$ by the isomorphism in Theorem \ref{thm:bea} is contained in $X$.  It seems to be an
interesting problem to investigate this map.
\end{remark}

%%%%%%%%%%%%%%%%%%%%%%%%%%%%%%%%%%%%%%%%%%%%%%%%%%%%%%%%%%
\section{Poisson structures and Hamiltonian vector fields}
%%%%%%%%%%%%%%%%%%%%%%%%%%%%%%%%%%%%%%%%%%%%%%%%%%%%%%%%%%

The goal of this section is to construct a (multi-) Hamiltonian structure on the quotient space
$M_{ir}^{\Delta_{\ue}}/G_{\Delta_{\ue}}$, by using the momentum map of $G_{\Delta_{\ue}}$ acting on
$M_{p,d+p}^{\Delta_{\ue}}$.  In what follows, we first recall the Hamiltonian structure on $M_{p,q}$, next we use
it to construct a (multi-) Hamiltonian structure on $M_{p,q}^{\Delta_{\ue}}$ which will finally lead to a (multi-)
Hamiltonian structure on $M_{ir}^{\Delta_{\ue}}/G_{\Delta_{\ue}}$.

\subsection{Multi-Hamiltonian structure on $M_{p,q}$} 
Let $p \geq 2, q \geq 1$ be integers.  We briefly recall the multi-Hamiltonian structure on $M_{p, q}$.  For
details, see \cite{RS94, PV, PLV}.

For $L(x) \in M_{p,q}$ and $i,j,k \in\bbZ,$ such that $1\leqs i,j\leqs p$, we write $\ell^k(L(x)) \in \Mat_p(\C)$
(resp. $\ell^k_{ij}(L(x)) \in \C$) for the coefficient of $x^k$ (resp. the $(i,j)$-th entry of $\ell^k(L(x))$).
The ring $\cO(M_{p,q})$ of regular functions on $M_{p,q}$ is generated (as a $\C$-algebra) by $\ell_{ij}^k, 1 \leq
i, j \leq p,~ 0 \leq k \leq q$.  We also put $\ell_{ij}^k = 0$ for $k > d$ and for $k < 0$.  The Poisson structures
which are given in the following proposition are restriction to $M_{p,q}$ of a family of linear Poisson structure
on the dual of the loop algebra of $\mathfrak{gl}(p,\mathbb{C})=\Mat_p(\C)$.

\begin{prop}[{\cite[eq.(10)]{RS88}, see also \cite{RS94}}]
\label{Lie-Poisson_loop}
For any integer $\mu$ with $0 \leq \mu \leq q+1$ there exists a unique Poisson bracket $\PB_\mu$ on $\cO(M_{p,q})$
for which
\begin{equation}\label{eq:poisson_loop}
  \pb{\ell_{ij}^k,\ell_{mn}^l}_\mu
  =
  \eta_\mu^{kl}(\ell_{mj}^{k+l+1-\mu}\delta_{ni}
  -\ell_{in}^{k+l+1-\mu}\delta_{jm})\;,
\end{equation}%
where $\eta_\mu^{kl}:=1$ if $k,l<\mu$ and $\eta_\mu^{kl}:=-1$ if $k,l\geqs \mu$ and $\eta_\mu^{kl}:=0$ otherwise.
Moreover, the brackets $\PB_\mu~ (\mu=0, \dots, q+1)$ form a family of compatible Poisson brackets, i.e., any
linear combination of these brackets is a Poisson bracket.
\end{prop}

Let $j$ and $i\geqs0$ be integers.  We define the {\it Hamiltonian function} $H_{ij} \in \cO(M_{p,q})$ by
\begin{align}\label{eq:Hamiltonian}
  H_{ij}(L(x))=\frac1{i+1}\Res_{x=0}{\Tr(L^{i+1}(x))\over x^{j+1}}\;
%=\frac1{i+1}\Tr( \ell^{j}(L^{i+1}(x)))
\end{align}
and the {\it Hamiltonian vector field} $\frac{\diff}{\diff t_{i,j}}$ on $M_{p,q}$ by
\begin{equation}\label{eq:ham_vf}
  \frac{\diff}{\diff t_{i,j}} :=-\Pb{H_{i,j}}_1\;. 
\end{equation}
Here we use the notation $\Res_{x=0} f(x) := f_{-1}$ for $f(x) = \sum_{k} f_k x^k \in \C[x,x^{-1}]$.  In what follows, for any
$L(x)=\sum_{k=0}^N L_k x^k \in \Mat_p(\C)[x]$ and $j \geq 0$, we write $(L(x)/x^{j})_+ := \sum_{k=0}^{N-j}
L_{k+j}x^k$.

\begin{remark}\label{rem:H-vf}
By the above definition, $H_{i,j} = 0$ for $j<0$ or $j>(i+1)q$.
%and $\frac{\diff}{\diff t_{i,j+1}} = 0$ for $j < 1$
%or $j>iq$.
\end{remark} 

\begin{prop} \cite[eqs.~(12)--(14)]{RS88} 
\label{prop:reduced-vf}
(1) For any integers $i$ and $j$ with $i\geqs 0$ and $0\leqs j\leqs(i+1)q$ the Hamiltonian vector field
(\ref{eq:ham_vf}) is given by the following Lax equation:
\begin{equation}\label{eq:vfs_on_loop}
  \frac{\diff}{\diff t_{i,j}}L(x)
%  =  \frac12\lb{L(x),R\nabla H_{i,j-1}(L(x))}
=\lb{L(x),\left(\frac{L^i(x)}{x^{j}}\right)_+}\;.
\end{equation}
Moreover, 
for any integer $\mu$ with $0 \leq \mu \leq q+1$ we have that
\begin{align}\label{eq:multi-h-on-loop}
\frac{\diff}{\diff t_{i,j}} = -\Pb{H_{i,j-1+\mu}}_\mu\;.
\end{align}
In particular, each of the vector fields (\ref{eq:ham_vf}) 
is multi-Hamiltonian.
\\
(2) The Hamiltonian functions $H_{i,j}$ are pairwise in involution with respect to each one of the  Poisson
brackets $\PB_\mu$, where $0 \leq \mu \leq q+1$.
\\
(3) Let $\mu$ be such that $0 \leq \mu \leq q+1$. For any $i\geqs 0$, $H_{i,j}$ is a Casimir function of
$\PB_{\mu}$ when $j \in \{0,\ldots,\mu-1,iq+\mu,\dots,(i+1)q\}$.
\end{prop}

\begin{proof}
We give only a sketch of the proof, 
as all claims are essentially in \cite{RS88}.
\\
(1) We get \eqref{eq:Hamiltonian}--\eqref{eq:multi-h-on-loop}
from the formulae in \cite{RS88} by setting
$P(L(x)) := \Tr L(x)^{i+1}/(i+1)$, $\varphi(x) = 1/x^{j+1}$
and $q(x) = x^{\mu}$.
\\
(2) It immediately follows from (1); for $i$, $j$, $\mu$ as above 
and $k \geq 1$ we have
$$
\{\Tr L(x)^k, ~H_{i,j-1+\mu} \}_\mu 
\stackrel{\eqref{eq:multi-h-on-loop}}{=} 
\frac{\diff \Tr L(x)^k}{\diff t_{i,j}} 
= k \Tr \left( L(x)^{k-1} \frac{\diff L(x)}{\diff t_{i,j}}\right)
\stackrel{\eqref{eq:vfs_on_loop}}{=} 0\;.
$$ 
\\ 
(3) It is clear from the Lax equation \eqref{eq:vfs_on_loop} that the vector field $\diff / \diff t_{i,j}$ is
identically zero when $i = 0$ or $j < 1$ or $j \geq iq+1$.  Then the claim follows from the correspondence between
\eqref{eq:vfs_on_loop} and \eqref{eq:multi-h-on-loop}.
\end{proof}

By the last part of Prop.~\ref{Lie-Poisson_loop}, if $\phi(x)=\sum_{\mu=0}^{q+1} c_\mu x^\mu$ is any polynomial of
degree at most $q+1$, then
\begin{equation*}
  \PB_\phi:=\sum_{\mu=0}^{q+1} c_\mu \PB_\mu\;
\end{equation*}
defines a Poisson structure on $M$. A convenient formula for these Poisson structures can be written down in terms
of generating functions, defined as follows. Let $x$ be a formal variable, define
\begin{align*}
  &\ell_{ij}(x) := \sum_{k=0}^{q+1} \ell_{ij}^k x^k\;,
\end{align*}
and let, for $\ast = \phi, ~\mu ~(0 \leqs \mu \leqs q+1)$ 
\begin{align*}
\pb{\ell_{ij}(x),\ell_{mn}(y)}_\ast:= \sum_{k,l=0}^q\pb{\ell_{ij}^k,\ell_{mn}^l}_\ast x^ky^l.
\end{align*}

\begin{prop}\label{prop:Poisson-M}
Let $\phi(x)=\sum_{\mu=0}^{q+1} c_\mu x^\mu$ be a polynomial of degree at most $q+1$.  For any $1 \leq i, j, m, n
\leq p$, we have
\begin{align}\label{eq:pb_with_pols}
  \begin{split}
  &\pb{\ell_{ij}(x),\ell_{mn}(y)}_\phi
  \\
  & \quad = 
  \frac{\ell_{in}(x)\phi(y)-\ell_{in}(y)\phi(x)}{x-y}\delta_{jm}
  -\frac{\ell_{mj}(x)\phi(y)-\ell_{mj}(y)\phi(x)}{x-y}\delta_{ni}\;.
  \end{split}
\end{align}
\end{prop}

For its proof, we will use the following lemma.
\begin{lemma}
  Let $\xi^0,\xi^1,\dots,\xi^q$ be arbitrary (commuting, independent) variables. Set $\xi^i:=0$ for $i>q$ and for
  $i<0$. Denote $\xi(x):= \sum_{i=0}^{q}\xi^ix^i$. 
 Then for any $s$ with $0\leqs s \leqs q+1$,
\begin{equation}\label{eq:poisson_loop_poly}
  \left(\sum_{k,l=0}^{s-1}-\sum_{k,l=s}^{q}\right)\xi^{k+l+1-s}x^ky^l=\frac{\xi(y)x^s-\xi(x)y^s}{x-y}\;.
\end{equation}%
\end{lemma}
\begin{proof}
We multiply the left hand side in (\ref{eq:poisson_loop_poly}) by $x-y$ and determine in the resulting polynomial
the coefficient of $\xi^{t+1-s}$ for $t=s-1,\dots,q+s-1$, the other variables $\xi^i$ being zero by
definition. For $t=s-1,s,\dots,2s-2$ the coefficient of $\xi^{t+1-s}$ is given by
\begin{equation*}
  \sum_{k,l=0\atop k+l=t}^{s-1}(x^{k+1}y^l-x^ky^{s+1})=x^s y^{t+1-s}-x^{t+1-s}y^s\;,
\end{equation*}%
while for $t=2s,2s+1,\dots,q+s-1$ it is given by
\begin{equation*}
  -\sum_{k,l=s\atop k+l=t}^{q}(x^{k+1}y^l-x^ky^{s+1})=x^s y^{t+1-s}-x^{t+1-s}y^s\;;
\end{equation*}%
for the remaining value of $t$, to wit $t=2s-1$, the coefficient of $\xi^{t+1-s}$ is zero, which again can be
written as $x^s y^{t+1-s}-x^{t+1-s}y^s$. Summing up, the left hand side of
(\ref{eq:poisson_loop_poly}), multiplied by $x-y$, is given by
\begin{equation*}
  \sum_{t=s-1}^{q+s-1}\xi^{t+1-s}(x^s y^{t+1-s}-x^{t+1-s}y^s)=\xi(y)x^s-\xi(x)y^s\;,
\end{equation*}%
as was to be shown.
\end{proof}

\begin{proof}[Proof of Proposition \ref{prop:Poisson-M}]
Using \eqref{eq:poisson_loop} and \eqref{eq:poisson_loop_poly}
we obtain
\begin{eqnarray*}
  \pb{\ell_{ij}(x),\ell_{mn}(y)}_\mu
  &=&\sum_{k,l=0}^q\pb{\ell_{ij}^k,\ell_{mn}^l}_\mu x^ky^l\\
  &=&\sum_{k,l=0}^q \eta_\mu^{kl}(\ell_{mj}^{k+l+1-\mu}\delta_{ni}
  -\ell_{in}^{k+l+1-\mu}\delta_{jm})x^ky^l\\
  &=&\left(\sum_{k,l=0}^{\mu-1}-\sum_{k,l=\mu}^{q}\right) 
  (\ell_{mj}^{k+l+1-\mu}\delta_{ni}
  -\ell_{in}^{k+l+1-\mu}\delta_{jm})x^ky^l\\
  &=&\frac{\ell_{in}(x)y^\mu-\ell_{in}(y)x^\mu}{x-y}\delta_{jm}
  -\frac{\ell_{mj}(x)y^\mu-\ell_{mj}(y)x^\mu}{x-y}\delta_{ni}\;.
\end{eqnarray*}
Taking a linear combination of the above for $\mu=0,\ldots,q+1$,
we get \eqref{eq:pb_with_pols}.
\end{proof}

%%%%%%%%%%%%%%%%%
\subsection{Reduced Poisson structure on a fixed point locus}
%%%%%%%%%%%%%%%%%%%%

Set $\zeta:=e^{2\pi i/p}$ and fix $\ue \in (\Z / p \Z)^p$.
Recall that we have the automorphism
$\sigma_{\Delta_{\ue}}$ on $M_{p,q}$, where 
$\sigma_{\Delta_{\ue}}$ acts as
\begin{equation*}
  \sigma_{\Delta_{\ue}}(L(x))  = \Ad (\Delta_{\ue}^{-1}) (L(\zeta x))\;.
\end{equation*}
For $F \in \mathcal{O}(M_{p,q})$, we write $\sigma_{\Delta_{\ue}}^*(F)$ or $\sigma_{\Delta_{\ue}}^*F$ for $F\circ
\sigma_{\Delta_{\ue}}$. Then, we have
%Then, since $\sigma_r(E_{ij}x^k)=\zeta^{k+r(j-i)}E_{ij}x^k$, we have that 
%
\begin{equation}\label{eq:action_on_xis}
  (\sigma_{\Delta_{\ue}}^*) (\ell_{ij}^k)
  = \zeta^{k+e_j-e_i} \, \ell_{ij}^k\;.
\end{equation}%

We determine in the following proposition 
for which values of $\mu$ the map $\sigma_{\Delta_{\ue}}$ is a Poisson
automorphism of $(M_{p,q},\PB_{\mu})$.

\begin{prop}\label{prop:reduced-PB}
  Let $\mu$ be an integer such that $0 \leq \mu \leq q+1$. 
  The map $\sigma_{\Delta_{\ue}}$ is a Poisson automorphism 
  of $(M_{p,q},\PB_\mu)$ 
  if and only if 
  $\mu\equiv1\pmod p$. % , independently of $\ue$ .
\end{prop}
\begin{proof}
Since $\sigma_{\Delta_{\ue}}$ is an automorphism of $M_{p,q}$, it is a Poisson automorphism of $(M_{p,q},\PB_\mu)$
if and only if
\begin{equation}\label{eq:poisson_auto}
  \pb{\sigma_{\Delta_{\ue}}^*\ell_{ij}^k,
      \sigma_{\Delta_{\ue}}^*\ell_{mn}^l}_\mu
  =\sigma_{\Delta_{\ue}}^*\pb{\ell_{ij}^k,\ell_{mn}^l}_\mu\;,
\end{equation}%
for all $1 \leq i,j,m,n \leq p$ and $0 \leq k,l \leq q$. 
Since the functions $\ell_{ij}^k$ are eigenfunctions of $\sigma_{\Delta_{\ue}}$ and since the right hand
side of (\ref{eq:poisson_loop}) is zero, except when $j=m$ or $i=n$, the only restrictions on $\mu$ coming from
(\ref{eq:poisson_auto}) will come from these cases. Let us look at the case $j=m$, $i\neq n$. Then the left hand
side of (\ref{eq:poisson_auto}) is given by
\begin{equation*}
  \zeta^{k+l-e_i+e_n}\pb{\ell_{ij}^k,\ell_{jn}^l}_\mu
  = - \eta_\mu^{kl} \zeta^{k+l-e_i+e_n}\ell_{in}^{k+l+1-\mu}\;,
\end{equation*}%
while the right hand side is given by
\begin{equation*}
  - \eta_\mu^{kl}\sigma_{\Delta_{\ue}}^*\ell_{in}^{k+l+1-\mu}
  = - \eta_\mu^{kl}\zeta^{k+l+1-\mu-e_i+e_n}\ell_{in}^{k+l+1-\mu}\;.
\end{equation*}%
Comparing these two expressions, we see that $\mu\equiv1\pmod p$. 
One finds the same result in the case $i=n$.
\end{proof}

For $1 \leq i,j \leq p$ and $0 \leq k \leq q$, we define a function $b_{ij}^k \in
\mathcal{O}({M_{p,q}^{\Delta_{\ue}}})$ by
\begin{align}\label{eq:b-def}
  b_{ij}^k := \ell_{ij}^k|_{M_{p,q}^{\Delta_{\ue}}}\;,
\end{align}
and we set 
$$
  I_{p,q}^{\ue} := \{(i,j,k) ~|~ 1\leq i,j \leq p,~0 \leq k \leq q, ~k \equiv e_i-e_j\pmod p \}.
$$ 
Then $b_{ij}^k = 0$ if $(i,j,k) \notin I_{p,q}^{\ue}$, and $\{b_{ij}^k ~|~ (i,j,k) \in I_{p,q}^{\ue} \}$ gives a
system of coordinates on $M_{p,q}^{\Delta_{\ue}}$.  On the other hand, we define functions $\tilde{b}_{ij}^k \in
\mathcal{O}(M_{p,q})$ by
\begin{align}\label{eq:t-b}
  \tilde{b}_{ij}^k
  := 
  \frac{1}{p} \sum_{m=1}^p (\sigma_{\Delta_{\ue}}^\ast)^m 
  (\ell_{ij}^k)\;. %\Big|_{M_{p,q}^{\Delta_{\ue}}}.
\end{align}
By construction, $\tilde{b}_{ij}^k$ is $\sigma_{\Delta_{\ue}}$-invariant and is an extension of ${b}_{ij}^k$ to
$M_{p,q}$ for all $i,j,k$. 
Consequently, we can write $\tilde{b}_{ij}(x) = x^{e_i-e_j} \alpha_{ij}(x^p)$ for some $\alpha_{ij}(x) \in
\mathcal{O}(M_{p,q})[x]$ where $\tilde{b}_{ij}(x) := \sum_{k=0}^q \tilde{b}_{ij}^k x^k$.

According to Prop.~\ref{prop:reduced-PB} and \cite[Sect.\ 5.4.3]{PLV}, the fixed point locus
$M_{p,q}^{\Delta_{\ue}}$ of $\sigma_{\Delta_{\ue}}$ inherits a Poisson structure from $\PB_\mu$ when $\mu \equiv
1\pmod p$, which is characterized by (\cite[Prop.~5.36]{PLV})
\begin{align}\label{eq:red-poisson}
  \{ b_{ij}(x) , b_{mn}(y) \}_\mu^{\text{red}}
  =
  \{\tilde{b}_{ij}(x), 
     \ell_{mn}(y)\}_\mu \big|_{M_{p,q}^{\Delta_{\ue}}}\;,
\end{align}
where we write $b_{ij}(x) = \sum_{k=0}^q {b}_{ij}^k x^k$.
It is explicitly written as follows:
\begin{prop}
Let $\phi(x) = x \tilde{\phi}(x^p)$ be a polynomial of degree at most $q+1$, where $\tilde{\phi}(x) \in \bbC[x]$.
The Poisson bracket \eqref{eq:red-poisson} on $M_{p,q}^{\Delta_{\ue}}$ is written as
\begin{align}\label{eq:b-Poisson}
  \begin{split}
  &\{ b_{ij}(x) , b_{mn}(y) \}_{\phi}^{\mathrm{red}}
  \\
  & \quad =
  \delta_{jm} 
  \frac{b_{in}(x) x^{p-\modp{e_m-e_n}} y^{\modp{e_m-e_n}-1} \phi(y)
        - b_{in}(y) y^{\modp{e_j-e_i+1}-1} x^{p-\modp{e_j-e_i+1}} 
          \phi(x)}
  {x^p - y^p}
  \\
  & \qquad - 
  \delta_{in} 
  \frac{b_{mj}(x) x^{p-\modp{e_m-e_n}} y^{\modp{e_m-e_n}-1} \phi(y) 
        - b_{mj}(y) y^{\modp{e_j-e_i+1}-1} x^{p-\modp{e_j-e_i+1}}
          \phi(x)}
  {x^p - y^p}
  \end{split}
\end{align}
where we set 
$$
  \modp{i} 
  = 
  \begin{cases} i & (1 \leq i \leq p) \\
                p+i & (-p \leq i \leq 0)\;.
  \end{cases} 
$$
\end{prop}

\begin{proof}
Let $\mu$ be a positive integer with $\mu\leq q+1$ and $\mu\equiv 1\pmod p$.
We calculate the r.h.s of \eqref{eq:red-poisson} 
using \eqref{eq:action_on_xis} and \eqref{eq:t-b}:
\begin{align*}
  r.h.s 
  &= 
  \frac{1}{p}\sum_{k=1}^p \zeta^{k(e_j-e_i)}
  \{\ell_{ij}(\zeta^k x), \ell_{mn}(y)\}_\mu   
  \big|_{M_{p,q}^{\Delta_{\ue}}}
  \\
  &\stackrel{\eqref{eq:pb_with_pols}}{=}
  \frac{1}{p}\sum_{k=1}^p \zeta^{k(e_j-e_i)}
  \left( \delta_{mj} \frac{\ell_{in}(\zeta^k x)y^\mu 
                           - \ell_{in}(y) (\zeta^k x)^\mu}{\zeta^k x -y} 
   - \delta_{in} \frac{\ell_{mj}(\zeta^k x) y^\mu 
                       - \ell_{mj}(y) (\zeta^k x)^\mu}{\zeta^k x -y} 
  \right) \Big|_{M_{p,q}^{\Delta_{\ue}}}
  \\
  &=
  \frac{1}{p}\sum_{k=1}^p \zeta^{k(e_j-e_i)}   
  \left( \delta_{mj} \frac{b_{in}(\zeta^k x)y^\mu 
                           - b_{in}(y) \zeta^k x^\mu}{\zeta^k x -y} 
   - \delta_{in} \frac{b_{mj}(\zeta^k x) y^\mu 
                       - b_{mj}(y) \zeta^k x^\mu}{\zeta^k x -y} 
  \right) 
  \\
  &=
  \frac{\delta_{jm}}{p} \left(b_{in}(x) y^\mu \sum_{k=1}^p 
                     \frac{\zeta^{k(e_j-e_n)}}{\zeta^k x -y}  
         -   b_{in}(y) x^\mu \sum_{k=1}^p 
                          \frac{\zeta^{k(e_j-e_i+1)}}{\zeta^k x -y} 
         \right)
  \\
  & \qquad -
  \frac{\delta_{in}}{p} \left(b_{mj}(x) y^\mu \sum_{k=1}^p 
                     \frac{\zeta^{k(e_m-e_i)}}{\zeta^k x -y}  
         -   b_{mj}(y) x^\mu \sum_{k=1}^p 
                          \frac{\zeta^{k(e_j-e_i+1)}}{\zeta^k x -y} 
         \right)\;. 
\end{align*}
By using Lemma \ref{lem:p-sum} below we get
\begin{align*}
  r.h.s. = \,
  &\delta_{jm} 
  \frac{b_{in}(x) x^{p-\modp{e_m-e_n}} y^{\modp{e_m-e_n}-1+\mu}
        - b_{in}(y) y^{\modp{e_j-e_i+1}-1} x^{p-\modp{e_j-e_i+1}+\mu}}
  {x^p - y^p}
  \\
  & \quad - 
  \delta_{in} 
  \frac{b_{mj}(x) x^{p-\modp{e_m-e_n}} y^{\modp{e_m-e_n}-1+\mu} 
        - b_{mj}(y) y^{\modp{e_j-e_i+1}-1} x^{p-\modp{e_j-e_i+1}+\mu}}
  {x^p - y^p}\;.
\end{align*}
This proves the formula when $\phi(x)=x^\mu$, with $\mu$ as above. Taking a linear combination over all $x^\mu$
which appear in $\phi(x)$ leads to the general formula.
\end{proof}

\begin{lemma}\label{lem:p-sum}
For $l=1,\ldots,p$, we have 
\begin{align*}
  \sum_{k=1}^p \frac{\zeta^{kl}}{t-\zeta^k} 
  =
  \frac{p \, t^{l-1}}{t^p-1}\;.
\end{align*}
\end{lemma}

\begin{proof}
We consider a partial fraction expansion of r.h.s:
$$
  \frac{p \, t^{l-1}}{t^p-1}
  =
  \sum_{k=1}^p \frac{a_k}{t-\zeta^k}\;.
$$
By multiplying $(t-\zeta^k)$ and setting $t=\zeta^k$ 
for $k=1,\ldots,p$, we get
$$
  a_k=\lim_{t \to \zeta^k} \frac{p \,t^{l-1} (t-\zeta^k)}{t^p-1}
  =
  \lim_{t \to \zeta^k} \frac{p \, t^{l-1}}{p \, t^{p-1}}
  =
  \zeta^{(l-p)k} 
$$
Thus the claim follows.   
\end{proof}

From Prop.~\ref{prop:reduced-vf}  
and Prop.~\ref{prop:reduced-PB},
we obtain the following result:
\begin{prop}\label{prop:reduced-HS}
(1) On $M_{p,q}^{\Delta_{\ue}}$, 
for integers $i,m,n$ such that 
$0 \leq i$ and
$0 \leq n \leq \lfloor q/p \rfloor$, 
%$ 0 \leq m +n \leq (i+1) \lfloor q/p \rfloor$,
we have the Hamiltonian vector fields given by
\begin{align}\label{eq:multi-h-reduced}
  \frac{\diff}{\diff t_{i,pm}} 
  = 
  -\{ \cdot ~, H_{i,p(m+n)}\}_{1+pn}^{\mathrm{red}}\;.
\end{align}
Here the Hamiltonian function $H_{i,pj}$ on $M_{p,q}^{\Delta_{\ue}}$
is given by
\begin{align}\label{eq:H-Mpq}
  H_{i,pj}(B(x))
  =
  \frac1{i+1}\Res_{x=0}{\Tr(B^{i+1}(x))\over x^{pj+1}}\;
\end{align}
for $B(x) = (b_{kl}(x))_{1 \leq k,l \leq p} 
\in M_{p,q}^{\Delta_{\ue}}$ .
Moreover, we have the Lax equation for \eqref{eq:multi-h-reduced}:
\begin{equation}\label{eq:vfs_reduced}
  \frac{\diff}{\diff t_{i,pm}}B(x)
=\lb{B(x),\left(\frac{B^i(x)}{x^{pm}}\right)_+}\;.
\end{equation}
(2) For each $0 \leq n \leq \lfloor q/p \rfloor$, 
$H_{i,pj}$ are  
Casimir functions with respect to $\PB_{pn+1}^{\mathrm{red}}$
if $i \geq 0$ and 
$j \in \{ 0, \dots, n, 
i\lfloor q/p \rfloor + n +1, \dots, 
(i+1) \lfloor q/p \rfloor \}.$
\end{prop}

\subsection{Poisson structure on the quotient space}
\label{sect:poisson-quot}

We set $d = d' p$ and fix 
$\ue = (e_1,e_2,\cdots, e_p) \in \mathcal{E}$.
Now we construct the Poisson structure on
the quotient space $M_{ir}^{\Delta_{\ue}}/G_{\Delta_{\ue}}$
by using the {\it momentum map} 
of $G_{\Delta_{\ue}}$ acting on $M_{p,d+p}^{\Delta_{\ue}}$.
We use the following notations.
\begin{align}
\nonumber
&\widetilde{M}^{\Delta_{\ue}}
:= 
\{B(x) \in M_{p,d+p}^{\Delta_{\ue}} ~|~
  H_{0,d+p}(B(x)) = 0 \}
\subset M_{p,d+p}^{\Delta_{\ue}}\;,
\\
\label{eq:def-t0}
&T_0 := \{ (i,j) ~|~ 1 \leq i,j \leq p, ~ e_i = e_j \}\;.
%\\
%&T_+ := \{ (i,j) ~|~ 1 \leq i,j \leq p, ~ e_i < e_j \},
%\\
%&T_- := \{ (i,j) ~|~ 1 \leq i,j \leq p, ~ e_i > e_j \}.
\end{align}
Here $H_{0,d+p} \in \mathcal{O}(M_{p,d+p}^{\Delta_{\ue}})$
is the Hamiltonian function \eqref{eq:H-Mpq}.

We identify the Lie algebra $\mathrm{Lie}\,PGL_p(\C)$ of $PGL_p(\C)$ with $\Mat_p(\C)/\C \cdot \mathbb{I}$, and
regard the Lie algebra $\mathrm{Lie}\,G_{\Delta_{\ue}}$ of $G_{\Delta_{\ue}}$ as a subspace of it.

\begin{lemma}
We have
$$
  \mathrm{Lie}\,G_{\Delta_{\ue}} 
  =
  \langle E_{ij}; ~ (i,j) \in T_0 \rangle_{\C} 
  / \C \cdot \mathbb{I}_p\;. 
$$ 
\end{lemma} 

\begin{proof}
By the definition of $G_{\Delta_{\ue}}$ \eqref{eq:g-alpha},
a class of $x=(x_{ij})_{1 \leq i,j \leq p} \in \Mat_p(\C)$ in $\mathrm{Lie}\,PGL_p(\C)$ belongs to
$\mathrm{Lie}\,G_{\Delta_{\ue}}$ if and only if there exists $c \in \C$ such that $\Delta_{\ue}\, x
\,\Delta_{\ue}^{-1} = x + c \mathbb{I}$, but the l.h.s. is $(\zeta^{e_i - e_j} x_{ij})_{1 \leq i,j \leq p}$ hence
$c=0$ and $x_{ij} = 0$ for $e_i \neq e_j$.
\end{proof}

Since $\widetilde{M}^{\Delta_{\ue}}$ is an affine space, its tangent space at any point can be identified with
$\widetilde{M}^{\Delta_{\ue}}$.  Thus, for $E \in \mathrm{Lie} G_{\Delta_{\ue}}$, the linear map
$$
  X_E : \widetilde{M}^{\Delta_{\ue}}\to \widetilde{M}^{\Delta_{\ue}}\;,\qquad B \mapsto [B, E]
$$
can be regarded as a vector field on $\widetilde{M}^{\Delta_{\ue}}$.

We take $\phi(x) = x \tilde{\phi}(x^p)$ where $\tilde{\phi}(x) = \sum_{k=0}^{d'+1} c'_k x^k$ and $c'_{d'+1} \neq
0$.
We write
\begin{align}\label{eq:tilde-mu}
\widetilde{\mu}: \mathrm{Lie} G_{\Delta_{\ue}} \to
\mathcal{O}(\widetilde{M}^{\Delta_{\ue}})
\end{align}
for the linear map 
which sends the class of 
$E_{ij}$ in $\mathrm{Lie} G_{\Delta_{\ue}}$
to 
$$
[(\sum_k b_{mn}^k x^k)_{mn} \mapsto
b_{ji}^{d+p}/c'_{d'+1}] 
\in \mathcal{O}(\widetilde{M}^{\Delta_{\ue}})\;.
$$ Note that $\widetilde{\mu}$ is well-defined because $\widetilde{\mu}(\mathbb{I})(B) = H_{0,d+p}(B) = 0$ for any
$B \in \widetilde{M}^{\Delta_{\ue}}$.

\begin{prop}\label{prop:moment-map} 
(1) The map $\widetilde{\mu}$ \eqref{eq:tilde-mu} is a Lie algebra homomorphism, i.e., for any $E, F\in\mathrm{Lie} \,G_{\Delta_{\ue}}$,
\begin{align}\label{eq:Lie-homo}
  \{\widetilde{\mu}(E), \widetilde{\mu}(F)\}_{\phi}^{\mathrm{red}} = \widetilde{\mu}([E,F])\;.
%\{H_{E_{ij}}, H_{E_{kl}}\}_{\phi}^{\mathrm{red}} 
%= H_{[E_{ij},E_{kl}]}.
\end{align}
(2) For any $E \in \mathrm{Lie} G_{\Delta_{\ue}}$ and $B \in \widetilde{M}^{\Delta_{\ue}}$, we have
\begin{align}\label{eq:H-moment}
  X_{E}(B) 
   = - \left(\{ b_{ij}(x), 
~\widetilde{\mu}(E) \}_{\phi}^{\mathrm{red}}(B)
       \right)_{1 \leq i,j \leq p}\;,
\end{align}
where $b_{ij}(x) = \sum_{k} b_{ij}^k x^k$ (see \eqref{eq:b-def}).
\end{prop}

\begin{proof}
It is enough to check the claims for the generators of 
$\mathrm{Lie} \,G_{\Delta_{\ue}}$.
\\
(1) 
For $B \in \widetilde{M}^{\Delta_{\ue}}$, we have
$$ 
\{\widetilde{\mu}(E_{ij}), \widetilde{\mu}(E_{kl})\}_{\phi}^{\mathrm{red}}(B)
=
\frac{1}{(c'_{d'+1})^2} \{ b_{ji}^{d+p}, ~b_{lk}^{d+p} 
\}_{\phi}^{\mathrm{red}} 
\stackrel{\eqref{eq:b-Poisson}}{=}
\frac{1}{c'_{d'+1}} (\delta_{kj} b_{li}^{d+p} - \delta_{il} b_{jk}^{d+p})\;,
$$
and 
\begin{align*}
&\widetilde{\mu}([E_{ij},E_{kl}])(B) 
= \widetilde{\mu}(\delta_{jk} E_{il} - \delta_{il} E_{kj})(B)
= \delta_{jk} \,\widetilde{\mu}(E_{il})(B) 
  - \delta_{il} \, \widetilde{\mu}(E_{kj})(B)\;.
%= \frac{1}{c'_{d'+1}} (\delta_{jk} b_{li}^{d+p} - \delta_{il} b_{jk}^{d+p})\;.
\end{align*} 
These two coincide and the claim follows.
\\
(2) We write $B = (\beta_{ij})_{ij} \in \widetilde{M}^{\Delta_{\ue}}$.
We take $(n,m) \in T_0$
and calculate both sides of \eqref{eq:H-moment} for $E = E_{nm}$.
The l.h.s. of \eqref{eq:H-moment} is: 
\begin{align}\label{eq:B-E}
  X_{E_{nm}}(B) = [B,~E_{nm}] =\sum_{1\leq i \leq p} \beta_{in} E_{im} - \sum_{1\leq j \leq p} \beta_{mj} E_{nj}\;.   
\end{align}
We claim that for $1 \leq i,j \leq p$,
\begin{align}\label{eq:b-br}
\frac{1}{c'_{d'+1}} 
\{ b_{ij}(x), b_{mn}^{p+d} \}_{\phi}^{\mathrm{red}} 
= 
-\delta_{jm} b_{in}(x) + \delta_{in} b_{mj}(x)\;.
\end{align}
Then the $(i,j)$ entry of the r.h.s. of \eqref{eq:H-moment} 
is obtained by taking the negative of the value of \eqref{eq:b-br} at 
$B \in \widetilde{M}^{\Delta_{\ue}}$.
This coincides with the $(i,j)$ entry of \eqref{eq:B-E}.

We prove \eqref{eq:b-br}. For $1\leqs \a,\b\leqs p$, we define $f_{\a\b}\in
\mathcal{O}(\widetilde{M}^{\Delta_{\ue}})[x,y]$ by:
\begin{equation*}
  f_{\a\b}(x,y):=b_{\a\b}(x)x^{p-\modp{e_m-e_n}}y^{\modp{e_m-e_n}-1} \phi(y)-b_{\a\b}(y) y^{\modp{e_j-e_i+1}-1} x^{p-\modp{e_j-e_i+1}}\phi(x)\;.
\end{equation*}%
We compute the highest order term in $y$ of $f_{in}(x,y)$.  Set $T_+ := \{ (i,j) ~|~ 1 \leq i,j \leq p, ~ e_i < e_j
\}$ and $T_- := \{ (i,j) ~|~ 1 \leq i,j \leq p, ~ e_i > e_j \}$ and recall from (\ref{eq:def-t0}) the definition of
$T_0$.  First we assume $(i,m) \in T_0 \cup T_+$.  Then we have $(i,n) \in T_0 \cup T_+$, $\modp{e_m-e_n} = p$ and
$\modp{e_m-e_i+1} = e_m-e_i+1$, and we get
$$
f_{in}(x,y) = b_{in}(x) y^p \tilde{\phi}(y^p) - b_{in}(y) y^{e_m-e_i} x^{p-e_m+e_i}\tilde{\phi}(x^p)\;.
$$
Thus, the highest order term is $y^{d+2p} b_{in}(x) c'_{d'+1}$.
Next, we consider the case $(i,m) \in T_0$.
Since we have $(i,n) \in T_-$, 
$\modp{e_m-e_n} = p$ and $\modp{e_m-e_i+1} = p+e_m-e_i+1$.
Then we have 
$$
f_{in}(x,y) = b_{in}(x) y^p \tilde{\phi}(y^p) - b_{in}(y) y^{p+e_m-e_i} x^{-e_m+e_i}\tilde{\phi}(x^p)\;,
$$
which has as highest order term $y^{d+2p} b_{in}(x) c'_{d'+1}$.
In the same manner, 
we see that the highest order term in $y$ of $f_{mj}(x,y)$   
is $y^{d+2p} b_{mj}(x) c'_{d'+1}$. 
From \eqref{eq:b-Poisson}, we have
$$(x^p - y^p) \{ b_{ij}(x), b_{mn}(y) \}_{\phi}^{\mathrm{red}}
  = \delta_{jm} f_{in}(x,y) - \delta_{in} f_{mj}(x,y)\;. 
$$ 
By taking the coefficients of $y^{d+2p}$ in the above equality, we obtain \eqref{eq:b-br}.
\end{proof}

We define a map
$$
  \mu : ~\widetilde{M}^{\Delta_{\ue}} \to (\mathrm{Lie} G_{\Delta_{\ue}})^*
:= \mathrm{Hom}(\mathrm{Lie} G_{\Delta_{\ue}}, \C)
%  ~B(x) \mapsto \mu(B(x))
$$
by
$\mu(B)(E) := \widetilde{\mu}(E)(B)$
for all $B \in \widetilde{M}^{\Delta_{\ue}}$
and $E \in \mathrm{Lie} G_{\Delta_{\ue}}$.
The map $\widetilde{\mu}$ and $\mu$ are respectively called 
the {\it comomentum map} and the {\it momentum map} 
of $G_{\Delta_{\ue}}$ acting on $\widetilde{M}^{\Delta_{\ue}}$
(cf. \cite[\S 5.4.4]{PLV}).

We apply a general result on the momentum map 
(\cite[Prop.~5.39]{PLV}) to construct the Poisson bracket on
$M_{ir}^{\Delta_{\ue}}/G_{\Delta_{\ue}}$ induced by 
$\PB^{\text{red}}_{\phi}$  on $\widetilde{M}^{\Delta_{\ue}}$.
Let $\eta$ be a natural embedding 
$\eta : M^{\Delta_{\ue}}:= M_{p,d}^{\Delta_{\ue}}
 \hookrightarrow \widetilde{M}^{\Delta_{\ue}}$ 
given by $\eta(B) = B$. 
The image of $\eta$ is determined by the following conditions:
$B \in \Im \eta$ if and only if
$$
b^{p+d}_{ij}(B) = 0 \text{ for } (i,j) \in T_0,
\quad 
b^{p+d-k}_{i,j}(B) = 0 \text{ for } (i,j) \in T_{\pm}, 1\leq k \leq p-1\;.
$$
The first condition is nothing but 
the defining equation of $\mu^{-1}(0)$.
Note that $G_{\Delta_{\ue}}$ freely acts on $\eta(M_{ir}^{\Delta_{\ue}})$
and $\eta(M_{ir}^{\Delta_{\ue}}) \subset \mu^{-1}(0)$ is invariant 
under the action of $G_{\Delta_{\ue}}$.
By Prop.~\ref{prop:moment-map},
we can apply
the theory of Poisson reduction \cite[Prop.~5.39]{PLV},
and obtain the Poisson bracket on 
$\eta(M_{ir}^{\Delta_{\ue}})/G_{\Delta_{\ue}}$ induced by 
$\PB^{\text{red}}_{\phi}$ 
on $\widetilde{M}^{\Delta_{\ue}}$.
It is passed to the Poisson bracket on
$M_{ir}^{\Delta_{\ue}}/G_{\Delta_{\ue}}$.
We write $\PB_{\phi}^{\text{quo}}$ for the induced Poisson bracket 
on $M_{ir}^{\Delta_{\ue}}/G_{\Delta_{\ue}}$.

Let $i \geq 0$ and $m \geq 0$. %$1 \leq m \leq i d'$.
Due to the Lax equation \eqref{eq:vfs_reduced}, 
the restriction of the vector field $\diff/\diff t_{i,pm}$ on 
$\widetilde{M}^{\Delta_{\ue}}$ to $\Im \eta$
is tangent to $\Im \eta$. 
The induced vector field
$M_{ir}^{\Delta_{\ue}}/G_{\Delta_{\ue}}$ 
is denoted by $Y_{i,pm}$.
For a $G_{\Delta_{\ue}}$-invariant function
$f \in \cO(\widetilde{M}^{\Delta_{\ue}})$ on 
$\widetilde{M}^{\Delta_{\ue}}$,
we denote by $\overline{f}$ the corresponding function
on $M_{ir}^{\Delta_{\ue}}/G_{\Delta_{\ue}}$.

\begin{prop}\label{prop:hvf-quo}
Let $0 \leq n \leq d'+1$.
\\
(1) 
Let $i \geq 0$ and $m \geq 0$. %$1 \leq m \leq i d'$
Then $Y_{i,pm}$ is Hamiltonian, i.e., 
$$
  Y_{i,pm}
  = \{ \cdot ,~ \overline{H}_{i,p(m+n)} \}_{1+pn}^{\mathrm{quo}}\;,
$$
where $H_{i,p(m+n)}$ is introduced in \eqref{eq:H-Mpq}.
\\
(2) 
Suppose 
$0 \leq i$ and  
$j \in \{0, \ldots, n, id'+n, \ldots,(i+1) d'\}$.
Then
$\overline{H}_{i,pj}$ is a Casimir function 
with respect to $\PB_{1+pn}^{\mathrm{quo}}$.
\end{prop}

\begin{proof}
(1) The claim follows from the definition of $Y_{i,pm}$.
\\
(2) 
The case of $j \in \{0, \ldots,n\}$ follows from
Prop.~\ref{prop:reduced-HS} (2). 
Next, let $j \in \{id'+n+1, \ldots,(i+1) d'\}$. 
When $m \geq id'+1$, 
the restricted vector field $\diff/\diff t_{i,pm}|_{\Im \eta}$ 
on $\Im \eta$ is zero,
since we have $\left(\frac{B^i}{x^{pm}}\right)_+ = 0$
in \eqref{eq:vfs_reduced} for $B \in \Im \eta$.
Thus we have $Y_{i,pm} = 0$,
and $\overline{H}_{i,p(m+n)}$ is a Casimir function 
with respect to $\PB_{1+pn}^{\mathrm{quo}}$.
Finally we consider the case $j=id'+n$. 
The restricted vector field 
$\diff/\diff t_{i,pid'}|_{\Im \eta}$ 
is tangent to $G_{\Delta_{\ue}}$-orbits
in $\Im \eta$, since in \eqref{eq:vfs_reduced}
we have $\left(\frac{B^i}{x^{pid'}}\right)_+ \in 
\langle E_{ij}; ~ (i,j) \in T_0 \rangle_{\C}$
for $B \in \Im \eta$.
Thus $Y_{i,pid'} = 0$,
and $\overline{H}_{i,p(id'+n)}$ is a Casimir function 
with respect to $\PB_{1+pn}^{\mathrm{quo}}$.
\end{proof}

%%%%%%%%%%%%%%%%%%%%%%%%%%%%%%%%%%%%%%%%%%
\section{Algebraic complete integrability}
%%%%%%%%%%%%%%%%%%%%%%%%%%%%%%%%%%%%%%%%%%
\label{sec:aci}

\subsection{Definition of aci and Beauville's result}
\label{sect:aci-beauville}
The following definition is taken from
\cite{V, adlermoerbekevanhaecke2004}.

\begin{defn}\label{def:aci}
Let $(M,\PB)$ be a complex Poisson manifold of rank $2r$, and let $H : M \to \C^s$ be a morphism with $s = \dim M -
r$.  For each $i =1, \dots, s$, we write $H_i \in \cO(M)$ for the $i$-th component of $H$.  We say that $(M, \{
\cdot, \cdot \}, H)$ is a Liouville integrable system if $H_1, \dots, H_s$ are pairwise in involution and
independent.  We say that $(M, \{ \cdot, \cdot \}, H)$ is an algebraic completely integrable system if moreover for
generic $m \in \C^s$, each connected component of the fiber $H^{-1}(m)$ is an affine part of an Abelian variety and
the restriction of each one of the Hamiltonian vector fields $X_{H_i}$ to each one of these affine parts is a
translation invariant vector field on the Abelian variety.
\end{defn}

Let $p \geq2, d \geq 1$ be integers.  We define the elementary symmetric polynomial $e_k$ and the power sum $f_k$
in $p$ variables $x_i ~(1 \leq i \leq p)$ by
$$
e_k := \sum_{1 \leq i_1<i_2< \cdots <i_k \leq p} 
          x_{i_1} \cdots x_{i_k},~
\quad
f_k := \sum_{1 \leq i \leq p} x_i^k
\qquad \text{for}~k=1, \dots, p\;.
$$
Then we have
$
%  \C[x_1,\ldots,x_p]^{\mathcal{S}_p} = 
\C[e_1,\ldots,e_p] = \C[f_1,\ldots,f_p] $ as subrings of $\C[x_1, \dots, x_p]$.  For $k=1, \dots, p$, we write
$S_k$ for the unique polynomial in $p$ variables satisfying $e_k = S_k(f_1, \ldots, f_p)$.  Put $\C[x]_{\leq m} :=
\{ h(x) \in \C[x] ~|~ \deg(h) \leq m \}$.  We define isomorphisms
%$\psi : \oplus_{i=1}^p \C[x]_{\leq id} \to V_{p,d}$
$$
\psi : \oplus_{i=1}^p \C[x]_{\leq id} \cong V_{p,d}; \quad
  \psi((h_i(x))_{1 \leq i \leq p}):=
   y^p + \sum_{j=0}^{p-1} S_j(h_1(x), \dots, h_p(x))y^j
$$
(see \S \ref{sect:levelsets} for the definition of $V_{p, d}$)
and 
$$ \gamma : \oplus_{i=1}^p \C[x]_{\leq id} \cong \C^s; \quad
\gamma((\sum_{j=0}^{id} h_{ij}x^j)_{1 \leq i \leq d})
 := (h_{ij})_{1 \leq i \leq p, ~ 0 \leq j \leq id}
$$
with $s:=\sum_{i=0}^{p-1}((i+1)d+1)=\frac{1}{2}p(p+1)d+p$.

We have a well-defined morphism
$$
 \alpha : M_{p,d,\mathrm{ir}}/PGL_p(\C) 
 \to \oplus_{i=1}^p \C[x]_{\leq id};
\quad
\alpha([L(x)])
:=
\left(\frac{1}{i} \Tr L(x)^{i}\right)_{1 \leq i \leq p}\;,
$$
where $[L(x)]$ is the class of $L(x) \in M_{p,d,\mathrm{ir}}$
in $M_{p,d,\mathrm{ir}}/PGL_p(\C)$.
For a $PGL_p(\C)$-invariant function $f \in \mathcal{O}(M_{p,d+1})$,
we denote the corresponding function on 
$M_{p,d,ir}/PGL_p(\C)$ by $\widetilde{f}$. 
Then the components of the composition 
$\widetilde{H} := \gamma \circ \alpha :
M_{p,d,\mathrm{ir}}/PGL_p(\C) \to \C^s$ 
are given by
$\widetilde{H}_{ij} ~(0 \leq i \leq p-1, ~ 0 \leq j \leq (i+1)d)$,
where
$H_{ij}$ is given by \eqref{eq:Hamiltonian}.
The composition $\psi \circ \alpha$
agrees with $\tilde{\chi}$ 
defined in \eqref{eq:chitilde}.
We summarize this in 
the following lemma:

\begin{lemma}\label{lem:H-chi}
We have a commutative  diagram:
\begin{equation}
\xymatrix{
&M_{p,d,\mathrm{ir}}/PGL_p(\C)
\ar[dl]_{\tilde{\chi}}
\ar[d]^{\alpha} 
\ar[dr]^{\widetilde{H}}
& 
\\
V_{p,d} &
\ar[l]^{\psi}_{\cong}
\oplus_{i=1}^p \C[x]_{\leq id} 
\ar[r]_{\gamma}^{\cong}
&
\C^{s}.
}
\end{equation}
\end{lemma}

For each $\mu=0, 1, \dots, d+1$,
we write $\PB_\mu^B$ for 
the Poisson bracket on $M_{p,d,ir}/PGL_p(\C)$ 
induced from 
$\PB_\mu$ \eqref{eq:poisson_loop} on $M_{p,d+1}$
%through the momentum map of $PGL_p(\C)$ acting on $M_{p,d+1}$
(see \cite[\S 5]{B}, \cite[\S 2.2]{IK}). 
As we recalled in Theorem \ref{thm:bea},
the fiber $\widetilde{H}^{-1}(m)$ is isomorphic to 
$\Pic^{g_P-1}(C_P) \setminus \Theta_P$
for any $m \in \mathbb{C}^s$
such that $P = \psi \circ \gamma^{-1}(m) \in V_{p,d,\sm}$.
Based on this result, Beauville proved the following
fundamental theorem \cite[Theorem 5.3]{B}:

\begin{thm}[Beauville]
\label{th:Beauville}
For each $\mu=0, 1, \dots, d+1$, 
the triple $(M_{p,d,ir}/PGL_p(\C), \PB_\mu^B, \widetilde{H})$ is an
algebraic completely integrable system.
\end{thm}

\subsection{Main result}
We suppose from now on that $p$ is a prime number and that $d=d'p$ for some $d' \geq 1$.  Fix $\ue \in
\mathcal{E}$.  Similarly to Lemma \ref{lem:H-chi}, we shall construct the following commutative diagram:
\begin{equation}\label{eq:diag-ue}
\xymatrix{
&M_{p,d,\mathrm{ir}}^{\Delta_{\ue}}/G_{\Delta_{\ue}}
\ar[dl]_{\tilde{\chi}_{\ue}}
\ar[d]^{\alpha_{\ue}} 
\ar[dr]^{\overline{H}_{\ue}}
& 
\\
V_{p,d'} &
\ar[l]^{\overline{\psi}}_{\cong}
\oplus_{i=1}^p \C[x]_{\leq id'} 
\ar[r]_{\overline{\gamma}}^{\cong}
&
\C^{s'}
}
\end{equation}
with $s'=\sum_{i=0}^{p-1}((i+1)d'+1)=\frac{1}{2}p(p+1)d'+p$.
The isomorphisms $\overline{\psi}$ and $\overline{\gamma}$
are obtained by applying 
the constructions of $\psi$ and $\gamma$
for $d$ replaced by $d'$.
For any 
$B(x) \in M_{p,d,\mathrm{ir}}^{\Delta_{\ue}} \subset M_{p, d, \ir}$,
there exists a unique $Q(x, y) \in V_{p, d'}$
(resp. $H(x) \in \oplus_{i=1}^p \C[x]_{\leq id'}$)
such that $\chi([B(x)]) = Q(x^p, y)$
(resp. $\alpha([B(x)]) = H(x^p)$).
The image of the class of $B(x)$ in 
$M_{p,d,\mathrm{ir}}^{\Delta_{\ue}}/G_{\Delta_{\ue}}$
by the map $\tilde{\chi}_{\ue}$ (resp. $\alpha_{\ue}$)
is defined to be $Q(x, y)$ (resp. $H(x)$).
Finally, $\overline{H}_{\ue}$
is defined by
the collection of Hamiltonian functions 
$\overline{H}_{i,pj} ~(0 \leq i \leq p-1, ~ 0 \leq j \leq (i+1)d')$ 
introduced in \S \ref{sect:poisson-quot}.

For each $n=0, 1, \dots, d'+1$, 
we have defined the Poisson bracket 
$\PB^{\mathrm{quo}}_{1+pn}$ on 
$M^{\Delta_{\ue}}_{ir}/G_{\Delta_{\ue}}$ 
in \S \ref{sect:poisson-quot}.
As we proved in Theorem \ref{thm:each-conn-comp-jac},
the fiber $\overline{H}_{\ue}^{-1}(m)$ is isomorphic to 
the disjoint union of some affine subvarieties of $J(C_Q)$
for any $m \in \mathbb{C}^{s'}$
such that 
$Q = \overline{\psi}(\overline{\gamma}^{-1}(m)) \in V_{p,d',\spl}$.
Based on this result, 
we prove our main theorem:

\begin{thm}\label{thm:main}
For each $n=0, 1, \dots, d'+1$, 
the triple 
$$(M^{\Delta_{\ue}}_{p, d, \ir}/G_{\Delta_{\ue}}, 
\PB^{\mathrm{quo}}_{1+pn},\overline{H}_{\ue})
$$
is an
algebraic completely integrable system.
\end{thm}

To prove this theorem we use the following results.

\begin{lemma}\label{lem:H-ind}
For $0 \leq i \leq p-1,~ 0 \leq j \leq (i+1)d'$, the $G_{\Delta_{\ue}}$-invariant functions $H_{i,pj} \in
\mathcal{O}(M_{p, d,\mathrm{ir}}^{\Delta_{\ue}})$ are algebraically independent.
\end{lemma}

\begin{proof}
Due to Theorem \ref{thm:each-conn-comp-jac}, 
for $Q \in V_{p,d',\mathrm{spl}}$ 
the space $\chi^{-1}_{\ue}(Q) / G_{\Delta_{\ue}}$ is isomorphic to 
the disjoint union of affine open subsets of $J(C_Q)$.
In particular, $\chi_{\ue}$ is dominant
and we have
\begin{align}\label{eq:dim-chi}
\dim \chi^{-1}_{\ue}(Q) = g_Q + \dim G_{\Delta_{\ue}}
= \frac{1}{2}(p-1)(pd'-2) + |T_0| -1
\end{align}
(see \eqref{eq:def-t0} for the definition of $T_0$).
We also have
\begin{align}
  \label{eq:dim-V}
  &\dim V_{p,d'} = \sum_{k=1}^p (d'k+1) = \frac{1}{2} d'p(p+1) + p
  \\
  \label{eq:dim-M}
  &\dim M_{p, d, \ir}^{\Delta_{\ue}} = d'p^2 + |T_0|\;. 
\end{align}
From \eqref{eq:dim-chi}--\eqref{eq:dim-M} we obtain 
$\dim M_{p, d, \ir}^{\Delta_{\ue}} = \dim \chi^{-1}_{\ue}(Q) + \dim V_{p,d'}$,
and hence
$\dim M_{p, d, \ir}^{\Delta_{\ue}} = \dim \overline{H}^{-1}_{\ue}(m) + s'$
for generic $m \in \C^{s'}$ by \eqref{eq:diag-ue}.
Thus, since $\overline{H}_{\ue}$ is dominant, 
$H_{i,pj}$ are algebraically independent. 
\end{proof}

\begin{prop}\label{prop:vflinear-quo}
Fix $Q(x,y) \in V'_{p,d',\mathrm{spl}}$ and set $P(x,y) := Q(x^p,y) \in V_{p,d,\mathrm{spl}}$.  We write
$\widetilde{Y}_{i,j}$ for the Hamiltonian vector field on $M_{p,d,P}/PGL_p(\C)$, which is induced by $\diff/\diff
t_{i,j}$ on $M_{p,d}$.  Then every translation invariant vector field on $J(C_P)^\tau$ is a linear combination of
$\widetilde{Y}_{i,pj}$'s.
\end{prop}

\begin{proof}
It follows from Theorem \ref{th:Beauville} that $\widetilde{Y}_{i,j}$ is for 
$i,j \geq 0$ a translation invariant
vector field on $J(C_P)$.  From Theorem \ref{thm:each-conn-comp-jac} and Prop.~\ref{prop:hvf-quo} (1), we see that
$\widetilde{Y}_{i,pj}|_{J(C_P)^\tau}=Y_{i,pj}$ and 
it is tangent to $J(C_P)^\tau$.  
%By Prop.~\ref{prop:hvf-quo} and Lemma~\ref{lem:H-ind},
By Prop.~\ref{eq:multi-h-on-loop} (2) and Theorem ~\ref{th:Beauville},
the vector fields $Y_{i,pj} ~(1 \leq i \leq p-1,~ 1 \leq j \leq id'-1)$ 
turn out to be linearly independent.  
The number of such vector fields $\sum_{i=1}^{p-1} (id'-1)$ agrees with $g_Q$.  Thus the claim follows.
\end{proof}

\begin{proof}[Proof of Theorem~\ref{thm:main}]
First we calculate the rank $2r$ of 
the Poisson manifold $(M_{p, d, \ir}^{\Delta_{\ue}}/G_{\Delta_{\ue}}, 
\PB^{\mathrm{quo}}_{1+pn})$.
Let $d(\Cas)$
be the number of independent Casimirs with respect to 
$\PB^{\mathrm{quo}}_{1+pn}$.
For $Q \in V_{p,d',\mathrm{spl}}$, we have 
$$
  2 r + d(\Cas)
  \leqs \dim M_{p, d, \ir}^{\Delta_{\ue}}/G_{\Delta_{\ue}} 
  = g_Q + s'\;, 
$$
while Prop.~\ref{prop:hvf-quo} (ii) and Lemma~\ref{lem:H-ind} show $d(\Cas) \geq (d'+2)p-1$.  We obtain $g_Q \geq
r$.  On the other hand, we have $g_Q \leq r$ because the Hamiltonian vector fields span the $g_Q$-dimensional
tangent space of $\chi_{\ue}^{-1}(Q)$ for $Q \in V_{p,d',\mathrm{spl}}$.  Therefore we get $r=g_Q$.

Now it follows from Lemma~\ref{lem:H-ind} that $(M_{p, d, \ir}^{\Delta_{\ue}}/G_{\Delta_{\ue}},
\PB^{\mathrm{quo}}_{1+pn},\overline{H}_{\ue})$ is a Liouville integrable system.  Finally, the algebraic completely
integrability follows from Theorem \ref{thm:each-conn-comp-jac} and Prop.~\ref{prop:vflinear-quo}.
\end{proof}

%%%%%%%%%%%%%%%%%%%%%%%%%%%%%%%%%%%%%%%%%%%%%%%%%
\section{Appendix: Pull-back on Jacobian variety}
%%%%%%%%%%%%%%%%%%%%%%%%%%%%%%%%%%%%%%%%%%%%%%%%%

The aim of this appendix is to prove the following theorem.

\begin{thm}\label{thm:inj-jac}
Let $C$ and $C'$ be smooth projective irreducible curves over $\bbC$, 
and let $f: C \to C'$ be a finite morphism.
Suppose that the corresponding extension $\bbC(C)/\bbC(C')$ 
of function fields is a Galois extension of prime degree $p$. 
%
%We denote by $R \subset C'$ the set of branch points of $f$, 
%and let $N := |R|$.
%Then the pull-back $f^* : J(C') \to J(C)$ is not injective 
%if and only if $N=0$ (that is, $f$ is \'etale). 
%
Then the pull-back $f^* : J(C') \to J(C)$ is not injective 
if and only if $f$ is \'etale (that is, unramified everywhere). 
If this is the case, $\Ker(f^*)$ is a cyclic group of order $p$.
\end{thm}

\begin{proof}
%[Proof of Theorem \ref{thm:inj-jac}]
%
In the rest of this section, we use the conventions introduced in \ref{sect:convention}.  Let $T :=
\Gal(\C(C)/\C(C'))$ and $F:=\Ker(f^* : J(C') \to J(C))$.

First we assume $f$ is \'etale and we prove that $F$ is a cyclic group of order $p$.  Since $f$ is \'etale, we have
the Hochschild-Serre spectral sequence
$$ 
  E_2^{ij}=H^i(T, H^j_\et(C, \G_m)) \Rightarrow H^{i+j}_\et(C', \G_m)\;. 
$$
(Here and in the sequel $H_\et^*(-, -)$ denotes \'etale cohomology, while $H^*(T, -)$ denotes Galois cohomology.
We denote by $\G_m$ the \'etale sheaf represented by
the multiplicative group $\G_m$.)
  Recall that $H^1_\et(X, \G_m) = \Pic(X)$ for any scheme $X$.  Thus we get an exact sequence
$$ 0 \to H^1(T, \bbC^*) \to \Pic(C') \overset{f^*}{\to} \Pic(C)^T\;. $$
(Here we used $H^0_\et(C, \G_m)=\bbC^*$, which is a consequence of the assumption that $C$ is projective over
$\bbC$.)  We obtain $F \cong \mu_p$ by Lemmas \ref{lem:gal-1} and \ref{lem:pic}.

Next we assume $F \not= 0$ and show that $f$ is \'etale.  We write $J(C)[p]$ (resp. $J(C')[p]$) for the group of
$p$-torsion elements in $J(C)$ (resp. $J(C')$).  Since $f$ is of degree $p$, we have $F \subset J(C')[p]$.  (This
can be seen by $f_* \circ f^*(a)=pa$ for any $a \in J(C')$, where $f_* : J(C) \to J(C')$ denotes the the
push-forward map.)  On the other hand, we have two horizontal isomorphisms in a commutative diagram
$$
\begin{matrix}
H^1_\et(C', \mu_p) & \cong & J(C')[p]
\\[1mm]
{}^{f^*} \downarrow \quad
& &
\downarrow^{f^*}
\\
H^1_\et(C, \mu_p) & \cong & J(C)[p]
\end{matrix}
$$
arising from the short exact sequence
$0 \to \mu_p \to \G_m \overset{p}{\to} \G_m \to 0$.
Recall that 
$H^1(X, \mu_p) \cong \Hom(\pi_1(X), \mu_p)$
for any connected variety $X$ over $\C$.
Therefore we get
$$ 
  F \cong \Hom(\pi_1(C')/f_*(\pi_1(C)), \mu_p)\;. 
$$
Now the assumption $F \not= 0$ implies that $f$ has a non-trivial \'etale subcovering, but $f$ is of prime degree
$p$ and hence $f$ must be \'etale.
\end{proof}

\def\cprime{$'$}

\end{document}